\documentclass{amsart}

\usepackage{a4,amsmath,amssymb,amscd,verbatim}

\newtheorem{proposition}[equation]{Proposition}
\newtheorem{theorem}[equation]{Theorem}
\newtheorem{corollary}[equation]{Corollary}
\newtheorem{lemma}[equation]{Lemma}

\theoremstyle{definition}
\newtheorem{problem}[equation]{Problem}

\newtheorem{definition}[equation]{Definition}
\newtheorem{definitions}[equation]{Definitions}
\newtheorem{remark}[equation]{Remark}
\newtheorem{remarks}[equation]{Remarks}
\newtheorem{example}[equation]{Example}
\newtheorem{examples}[equation]{Examples}
\newtheorem{conjecture}[equation]{Conjecture}

\numberwithin{equation}{section}

\newcommand{\spandsp}{\mbox{$\qquad\text{and}\qquad$}}
\newcommand{\sporsp}{\mbox{$\qquad\text{or}\qquad$}}
\newcommand{\sands}{\mbox{$\quad\text{and}\quad$}}
\newcommand{\sors}{\mbox{$\quad\text{or}\quad$}}

\newcommand{\sts}[1]{\mbox{$\quad\text{#1}\quad$}}

\newcommand{\Hom}{\operatorname{Hom}}

\newcommand{\letbe}{\mathbin{\raisebox{.4pt}{:}\!=}}

\newcommand{\GL}{\mbox{${\it GL}$}}
\newcommand{\?}{\mskip-2mu}
\newcommand{\+}{\mskip-.5mu}

\newcommand{\£}{\mskip1mu}

\newcommand{\MU}{\mbox{\it MU\,}}

\newcommand{\SU}{\mbox{$\it SU$}}

\newcommand{\KO}{\mbox{\it KO\/}}

\newcommand{\BU}{\mbox{\it BU}}

\newcommand{\bC}{\mathbb{C}}

\newcommand{\bQ}{\mathbb{Q}}
\newcommand{\bR}{\mathbb{R}}

\newcommand{\bZ}{\mathbb{Z}}

\newcommand{\fix}{\operatorname{Fix}}

\newcommand{\ra}{\rightarrow}

\newcommand{\lra}{\longrightarrow}

\newcommand{\llra}{\relbar\joinrel\hspace{-1pt}\longrightarrow}
\newcommand{\lllra}{\relbar\joinrel\hspace{-1pt}\llra}
\newcommand{\llllra}{\relbar\joinrel\hspace{-1pt}\lllra}
\newcommand{\llla}{\longleftarrow\joinrel\hspace{-1pt}\relbar}
\newcommand{\lllla}{\llla\joinrel\hspace{-1pt}\relbar}

\newcommand{\brpm}{\mbox{$\bR_{\scriptscriptstyle\geqslant}^m$}}

\newcommand{\srf}[1]{\mbox{${\text{\it SR}^F}$}}

\newcommand{\cf}{\mbox{\it{cf\hspace{1pt}}}}

\newcommand{\daja}{Davis-Januszkiewicz}

\newcommand{\EP}{\mbox{$E\1P$}}

\newcommand{\zp}{\mathcal Z_P}

\def\geq{\geqslant}
\def\le{\leqslant}
\def\leq{\leqslant}
\newcommand{\lowidehat}[1]{\mbox{$\;\widehat{\raisebox{0ex}[1.3ex]{$\!\?#1$}}$}}

\begin{document}
\bibliographystyle{plain} \title[Toric genera]
{Toric genera}

\author{Victor M Buchstaber} \address{Steklov
  Mathematical Institute, Russian Academy of Sciences, Gubkina Street
  8, 119991 Moscow, Russia; \emph{and} School of Mathematics,
  University of Manchester, Oxford Road, Manchester M13 9PL, England}
\email{buchstab@mi.ras.ru, victor.buchstaber@manchester.ac.uk}
\author{Taras E Panov}
\address{Department of Mathematics and Mechanics, Moscow State
University, Leninskie Gory, 119992 Moscow, Russia; \emph{and}
Institute for Theoretical and Experimental Physics, Moscow 117259, Russia}
\email{tpanov@mech.math.msu.ru}

\author{Nigel Ray}
\address{School of
Mathematics, University of Manchester, Oxford Road, Manchester
M13 9PL, England}
\email{nigel.ray@manchester.ac.uk}

\thanks{The first and second authors were supported by the Russian
Foundation for Basic Research (grants 08-01-00541 and 08-01-91855KO)
and by the Russian State Programme for the Support of Leading
Scientific Schools (grant 1824.2008.1); the second author was also
supported by MIMS, and by the P~Deligne 2004 Balzan prize in
Mathematics. The third author was supported by the UK EPSRC (grant
EP/F060831/1) and by the Royal Society (grant R104514).}

\keywords{combinatorial data, complex cobordism, equivariant genera,
formal group laws, Hirzebruch genera, homotopical genera, isolated fixed
points, omniorientation, rigidity, quasitoric manifold}

\begin{abstract}
  Our primary aim is to develop a theory of equivariant genera for
  stably complex manifolds equipped with compatible actions of a torus
  $T^k$. In the case of omnioriented quasitoric manifolds, we present
  computations that depend only on their defining combinatorial data;
  these draw inspiration from analogous calculations in toric
  geometry, which seek to express arithmetic, elliptic, and associated
  genera of toric varieties in terms only of their fans. Our theory
  focuses on the universal toric genus $\varPhi$, which was introduced
  independently by Krichever and L\"offler in 1974, albeit from
  radically different viewpoints. In fact $\varPhi$ is a version of
  tom Dieck's bundling transformation of 1970, defined on
  $T^k$-equivariant complex cobordism classes and taking values in the
  complex cobordism algebra $\varOmega^*_U(BT^k_+)$ of the classifying
  space. We proceed by combining the analytic, the formal group
  theoretic, and the homotopical approaches to genera, and refer to
  the index theoretic approach as a recurring source of insight and
  motivation. The resultant flexibility allows us to identify several
  distinct genera within our framework, and to introduce parametrised
  versions that apply to bundles equipped with a stably complex
  structure on the tangents along their fibres. In the presence of
  isolated fixed points, we obtain universal localisation formulae,
  whose applications include the identification of Krichever's
  generalised elliptic genus as universal amongst genera that are
  rigid on $SU$-manifolds. We follow the traditions of toric geometry
  by working with a variety of illustrative examples wherever
  possible. For background and prerequisites we attempt to reconcile
  the literature of east and west, which developed independently for
  several decades after the 1960s.
\end{abstract}

\maketitle

%
%
%
%
%
%
%
%
%

\section{Introduction}\label{in}

The study of equivariant cobordism theory was begun by Conner and
Floyd during the early 1960s, in the context of smooth manifolds with
periodic diffeomorphisms \cite{co-fl64I}. They extended their ideas to
stably almost complex manifolds soon afterwards \cite{co-fl64II}, and
the subject has continued to flourish ever since. Beyond 1965, however,
two schools of research emerged, whose interaction was curtailed by
difficulties of communication, and whose literature is only now being
reconciled.

In Moscow, Novikov \cite{novi67}, \cite{novi68} brought the theory
of formal group laws to bear on the subject, and provided a
beautiful formula for invariants of the complex cobordism class of
an almost complex $\bZ/p\,$-manifold with isolated fixed points.
He applied his formula to deduce relations in the complex
cobordism algebra $\varOmega^*_U(B\bZ/p)$ of the classifying space
of $\bZ/p$, which he named the {\it Conner-Floyd relations}. These
were developed further by Kasparov \cite{kasp69}, Mischenko
\cite{misc69}, Buchstaber and Novikov \cite{bu-no71}, Gusein-Zade
and Krichever \cite{gu-kr73}, and Panov \cite{pano98I},
\cite{pano98M}, and extended to circle actions by Gusein-Zade
\cite{guza71M}, \cite{guza71I}. Generalisation to arbitrary Lie
groups quickly followed, and led to \emph{Krichever's formula}
\cite[(2.7)]{kric74} for the universal toric genus $\varPhi(N)$ of
an almost complex manifold. Subsequently \cite{kric90}, Krichever
applied complex analytic techniques to the study of formal power
series in the cobordism algebra $\varOmega^*_U(BT^k)$, and deduced
important rigidity properties for equivariant generalised elliptic
genera.

Related results appeared simultaneously in the west, notably by tom
Dieck \cite{tomd70} in 1970, who developed ideas of Boardman
\cite{boar67} and Conner \cite{conn67} in describing a \emph{bundling
  transformation} for any Lie group $G$. These constructions were also
inspired by the $K$-theoretic approach of Atiyah, Bott, Segal, and
Singer \cite{at-bo68}, \cite{at-se68}, \cite{at-si68} to the index
theory of elliptic operators, and by applications of Atiyah and
Hirzebruch \cite{at-hi70} to the signature and
$\lowidehat{A}$-genus. Quillen immediately interpreted tom Dieck's
work in a more geometrical context \cite{quil71}, and computations
of various $G$-equivariant complex cobordism rings were begun soon
afterwards. Significant advances were made by Kosniowski
\cite{kosn76}, Landweber \cite{land72}, and Stong \cite{ston70}
for $\bZ/p$, and by Kosniowski and Yahia \cite{ko-ya83} for the
circle, amongst many others. In 1974, L\"offler \cite{loff74}
focused on the toric case of tom Dieck's transformation, and
adapted earlier results of Conner and Floyd on stably complex
$T^k$-manifolds to give a formula for $\varPhi(N)$; henceforth, we
call this \emph{L\"offler's formula}. More recent publications of
Comeza\~na \cite{come96} and Sinha \cite{sinh01} included
informative summaries of progress during the intervening years,
and showed that much remained to be done, even for abelian $G$. In
2005 Hanke \cite{hank05} confirmed a conjecture of Sinha for the
case $T^k$, and revisited tom Dieck's constructions. Surprisingly,
there appears to be no reference in these works to the
Conner-Floyd relations or to Krichever's formula (or even
L\"offler's formula!) although several results are closely
related, and lead to overlapping information.

A third strand of research was stimulated in Japan by the work of
Hattori \cite{hatt78} and Kawakubo \cite{kawa80}, but was always well
integrated with western literature.

One of our aims is to discuss $\varPhi$ from the perspective of both
schools, and to clarify relationships between their language and
results. In so doing, we have attempted to combine aspects of their
notation and terminology with more recent conventions of equivariant
topology. In certain cases we have been guided by the articles in
\cite{mayetal96}, which provide a detailed and coherent overview,
and by the choices made in \cite{bu-ra08}, where some of our results
are summarised. We have also sought consistency with \cite{bu-ra07},
where the special case of the circle is outlined.

As explained by Comeza\~na \cite{come96}, $G$-equivariant complex
bordism and cobordism occur in several forms, which have occasionally
been confused in the literature. The first and second forms are
overtly geometric, and concern smooth $G$-manifolds whose equivariant
stably complex structures may be normal or tangential; the two
possibilities lead to distinct theories, in contrast to the
non-equivariant case. The third form is homotopical, and concerns the
$G$-equivariant Thom spectrum $MU_G$, whose constituent spaces are
Thom complexes over Grassmannians of complex representations. Finally,
there is the Borel form, defined as the non-equivariant complex
bordism and cobordism of an appropriate Borel construction. We
describe and distinguish between these forms as necessary below. They
are closely inter-connected, and tom Dieck's bundling transformation
may be interpreted as a completion procedure on any of the first
three, taking values in the fourth.

Following L\"offler, and motivated by the emergence of toric
topology, we restrict attention to the standard $k$--dimensional
torus $T^k$, where $T$ is the multiplicative group of unimodular
complex numbers, oriented anti-clockwise. We take advantage of tom
Dieck's approach to introduce a generalised genus $\varPhi_X$, which
depends on a compact $T^k$-manifold $X$ as parameter space.
Interesting new examples ensue, that we shall investigate in more
detail elsewhere; here we focus mainly on the situation when $X$ is
a point, for which $\varPhi_X$ reduces to $\varPhi$. It takes values
in the algebra $\varOmega_*^U[[u_1,\dots,u_k]]$ of formal power
series over the non-equivariant complex cobordism ring, and may be
evaluated on an arbitrary stably complex $T^k$-manifold $N$ of
normal or tangential form. The coefficients of $\varPhi(N)$ are also
universal invariants of $N$, whose representing manifolds were
described by L\"offler \cite{loff74}; their origins lie in work of
Conner and Floyd \cite{co-fl64I}. These manifolds continue to
attract attention \cite{sinh01}, and we express them here as total
spaces of bundles over products of the bounded flag manifolds $B_n$
of \cite{ray86} and \cite{bu-ra98}. The $B_n$ are stably complex
boundaries, and our first contribution is to explain the
consequences of this fact for the rigidity properties of equivariant
genera.

When $N$ has isolated fixed points $x$, standard localisation
techniques allow us to rewrite the coefficients of $\varPhi(N)$ as
functions of the \emph{signs} $\varsigma(x)$ and \emph{weights}
$w_j(x)$, for $1\leq j\leq k$. Our second contribution is to extend
Krichever's formula to this context; his original version made no
mention of the $\varsigma(x)$, and holds only for almost complex $N$
(in which case all signs are necessarily $+1$). By focusing on the
constant term of $\varPhi(N)$, we also obtain expressions for the
non-equivariant complex cobordism class $[N]$ in terms of local
data. This method has already been successfully applied to evaluate
the Chern numbers of complex homogeneous spaces in an arbitrary
complex-oriented homology theory \cite{bu-te08}.

Ultimately, we specialise to omnioriented quasitoric manifolds
$M^{2n}$, which admit a $T^n$-action with isolated fixed points and a
compatible tangential stably complex structure. Any such $M$ is
determined by its \emph{combinatorial data} \cite{b-p-r07}, consisting
of the face poset of a simple polytope $P$ and an integral
dicharacteristic matrix $\varLambda$. Our third contribution is to
express the signs and weights of each fixed point in terms of
$(P,\varLambda)$, and so reduce the computation of $\varPhi(M)$ to a
purely combinatorial affair. By way of illustration, we discuss
families of examples whose associated polyhedra are simplices or
hypercubes; these were not available in combinatorial form before the
advent of toric topology.

Philosophically, four approaches to genera have emerged since
Hirzebruch's original formulation, and we hope that our work
offers some pointers towards their unification. The analytic and
function theoretic viewpoint has prospered in Russian literature,
whereas the homotopical approach has featured mainly in the west;
the machinery of formal group laws, however, has been well-oiled
by both schools, and the r\^ole of index theory has been enhanced
by many authors since it was originally suggested by Gelfand
\cite{gelf60}. Ironically, the dichotomy between the first and
second is illustrated by comparing computations of the Chern-Dold
character \cite{buch70} with applications of the Boardman
homomorphism and Hurewicz genus \cite{ray72}. In essence, these
bear witness to the remarkable influence of Novikov and Adams
respectively.

Our work impinges least on the index-theoretic approach, to which
we appeal mainly for motivation and completeness. Nevertheless, in
Section \ref{geri} we find it instructive to translate Krichever's
description of classical rigidity into the universal framework.
The relationship between indices and genera was first established
by the well-known Hirzebruch–-Riemann-Roch theorem~\cite{hirz66},
which expresses the index of the twisted Dolbeault complex on a
complex manifold $M$ as a characteristic number
$\langle\mathop{ch}\xi\mathop{td}\,(M),\sigma ^H_M\rangle$, where
$\xi$ is the twisting bundle; in particular, the untwisted case
yields the Todd genus. The cohomological form \cite{at-si68b} of
the Atiyah-Singer index theorem~\cite{at-si68} leads to explicit
descriptions of several other Hirzebruch genera as indices. For
example, the $\lowidehat{A}$-genus is the index of the Dirac
operator on Spin manifolds, and the signature is the index of the
signature operator on oriented manifolds.

Before we begin it is convenient to establish the following notation
and conventions, of which we make regular use.

All our topological spaces are compactly generated and weakly
Hausdorff \cite{vogt71}, and underlie the model category of
$T^k$-spaces and $T^k$-maps. It is often important to take account
of basepoints, in which case we insist that they be fixed by $T^k$;
for example, $X_+$ denotes the union of a $T^k$-space $X$ and a
disjoint fixed point $*$.

For any $n\geq0$, we refer to the action of $T^{n+1}$ by
coordinatewise multiplication on the unit sphere $S^{2n+1}\subset
\bC^{n+1}$ as the \emph{standard action}, and its restriction to the
diagonal circle $T_\delta<T^n$ as the \emph{diagonal action}. The
orbit space $S^{2n+1}/T_\delta$ is therefore $\bC P^n$; it is a
toric variety with respect to the action of the quotient $n$-torus
$T^{n+1}/T_\delta$, whose fan is normal to the standard $n$-simplex
$\varDelta^n\subset\bR^n$. The $k$-fold product
\begin{equation}\label{prodsphs}
\varPi\?S(q)\;\letbe\;\prod^kS^{2q+1}\subset\mathbb{C}^{k(q+1)}
\end{equation}
of $(2q+1)$-spheres admits a diagonal action of $T^k$, whose orbit
space $\varPi\?P(q)$ is the corresponding $k$-fold product of $\bC
P^q$s. The induced action of the quotient $kq$-torus
$T^{k(q+1)}/T^k$ turns $\varPi\?P(q)$ into a product of toric
varieties, whose fan is normal to the $k$-fold product of simplices
$\varPi\?\varDelta^q\subset\bR^{kq}$.

Given a free $T^k$-space $X$ and an arbitrary $T^k$-space $Y$, we
write the quotient of $X\times Y$ by the $T^k$-action
\[
  t\cdot(x,y)\;=\;(t^{-1}x,ty)
\]
as $X\times_{T^k}Y$; it is the total space of a bundle over $X/T^k$,
with fibre $Y$. If $X$ is the universal contractible $T^k$-space
$ET^k$, then $ET^k\times_{T^k}Y$ is the \emph{Borel construction} on
$Y$, otherwise known as the \emph{homotopy quotient} of the action.
The corresponding bundle
\[
Y\lra ET^k\times_{T^k}Y\lra BT^k
\]
therefore lies over the classifying space $BT^k$.

We let $V$ denote a generic $|V|$-dimensional unitary representation
space for $T^k$, which may also be interpreted as a $|V|$-dimensional
$T^k$-bundle over a point. Its unit sphere $S(V)\subset V$ and
one-point compactification $V\subset S^V$ both inherit canonical
$T^k$-actions. We may also write $V$ for the
product $T^k$-bundle $X\times V\ra X$ over an arbitrary $T^k$-space
$X$, under the diagonal action on the total space.

For any commutative ring spectrum $E$, we adopt the convention that
the homology and cohomology groups $E_*(X)$ and $E^*(X)$ are
\emph{reduced} for all spaces $X$. So their unreduced counterparts
are given by $E_*(X_+)$ and $E^*(X_+)$. The \emph{coefficient ring}
$E_*$ is the homotopy ring $\pi_*(E)$, and is therefore given by
$E_*(S^0)$ or $E^{-*}(S^0)$; we identify the homological and
cohomological versions without further comment, and interpret
$E_*(X_+)$ and $E^*(X_+)$ as $E_*$-modules and $E_*$-algebras
respectively. The corresponding conventions apply equally well to
the equivariant cobordism spectra $\MU_{T^k}$, and the $T^k$-spaces
$X$ with which we deal.

Our ring spectra are also \emph{complex oriented}, by means of a
class $x^E$ in $E^2(\bC P^\infty)$ that restricts to a generator
of $E^2(\bC P^1)$. It follows that $E^*(\bC P^\infty_+)$ is
isomorphic to $E_*[[x^E]]$ as $E_*$-algebras, and that $E$-theory
Chern classes exist for complex vector bundles; in particular,
$x^E$ is the first Chern class $c_1^E(\zeta)$ of the
\emph{conjugate} Hopf line bundle $\zeta\letbe\bar \eta$ over $\bC
P^\infty$. The universal example is the complex cobordism spectum
$\MU$, whose homotopy ring is the complex cobordism ring
$\varOmega_*^U$. If we identify the Thom space of $\zeta$ with
$\bC P^\infty$, then $x^E$ may be interpreted as a Thom class
$t^E(\zeta)$, and extended to a universal Thom class
$t^E\colon\MU\ra E$.

Such spectra $E$ may be \emph{doubly oriented} by choosing distinct
orientations $x_1$ and $x_2$. These are necessarily linked by
mutually inverse power series
\begin{equation}\label{bmseries}
x_2\;=\;\sum_{j\geq 0}b^E_j\?x_1^{j+1}\spandsp x_1\;=\;\sum_{j\geq
0}m^E_j\?x_2^{j+1}
\end{equation}
in $E^2(\bC P^\infty)$, where $b^E_j$ and $m^E_j$ lie in $E_{2j}$
for $j\geq 1$; in particular, $b^E_0\!=m^E_0\!=1$. We abbreviate
\eqref{bmseries} whenever possible, by writing the series as
$b^E(x_1)$ and $m^E(x_2)$ respectively; thus $b^E(m^E(x_2))=x_2$
and $m^E(b^E(x_1))=x_1$. It follows that the Thom classes
$t_1(\zeta)$ and $t_2(\zeta)$ differ by a unit in $E^0(\bC
P^\infty_+)$, which is determined by \eqref{bmseries}. More
precisely, the equations
\[
  t_2(\zeta)\;=\;t_1(\zeta)\cdot b^E_+(x_1)\spandsp
  t_1(\zeta)\;=\;t_2(\zeta)\cdot m^E_+(x_2)
\]
hold in $E^2(\bC P^\infty_+)$, where $b^E_+(x_1)=b^E(x_1)/x_1$ and
$m^E_+(x_2)=m^E(x_2)/x_2$.

A smash product $E\wedge F$ of complex oriented spectra is doubly
oriented by $x^E\wedge 1$ and $1\wedge x^F$ in $(E\wedge F)^2(\bC
P^\infty)$; we abbreviate these to $x^E$ and $x^F$ respectively.
\begin{example}\label{hmuexa}
  The motivating example is that of $E=H$, the integral Eilenberg-Mac
  Lane spectrum, and $F=\MU$. Then
  \smash{$b_j\letbe b_j^{H\wedge M\?U}$} is the standard
  indecomposable generator of $H_{2j}(\MU)$, and $t^H$ is the
  Thom class in $H^0(\MU)$. The inclusion $h^{M\?U}\colon\MU\ra
  H\wedge\MU$ may also be interpreted as a Thom class, and induces the
  Hurewicz homomorphism $h^{M\?U}_*\colon\varOmega^U_*\ra H_*(\MU)$.
\end{example}

We refer readers to Adams \cite{adam74} for further details, and to
\cite{bu-ra98} for applications of the universal doubly oriented
example $\MU\wedge\MU$.

Several colleagues have provided valuable input and assistance
during the preparation of this work, and it is a pleasure to
acknowledge their contribution. Particular thanks are due to
Andrew Baker, who drew our attention to the relevance of the work
of Dold \cite{dold78}, and to Burt Totaro, whose comments on an
earlier version led directly to Theorem \ref{multandrig}.
%
%
%
%
%
%
%
%
%

\section{The universal toric genus}

In this section we introduce tom Dieck's bundling transformation
$\alpha$ in the case of the torus $T^k$, and explain how to convert it
into our genus $\varPhi$. We discuss the geometric and homotopical
forms of $T^k$-equivariant complex cobordism theory on which $\varPhi$
is defined, and the Borel form which constitutes its target. In
particular, we revisit formulae of L\"offler and Krichever for its
evaluation. We refer readers to the papers of Comeza\~na
\cite{come96}, Comeza\~na and May \cite{co-ma96}, Costenoble
\cite{cost96}, and Greenlees and May \cite{gr-ma96} for a
comprehensive introduction to equivariant bordism and cobordism
theory.

We begin by recalling the Thom $T^k$-spectrum $\MU_{T^k}$, whose
spaces are indexed by the complex representations of $T^k$. Each
$\MU_{T^k}(V)$ is the Thom $T^k$-space of the universal
$|V|$-dimensional complex $T^k$-vector bundle $\gamma_{|V|}$, and each
spectrum map $S^{W\ominus V}\wedge\MU_{T^k}(V)\ra\MU_{T^k}(W)$ is
induced by the inclusion $V<W$ of a $T^k$-submodule. The
\emph{homotopical form} $\MU^*_{T^k}(X_+)$ of equivariant complex
cobordism theory is defined for any $T^k$-space $X$ by stabilising the
pointed $T^k$-homotopy sets $[S^V\wedge X_+,\MU_{T^k}(W)]_{T^k}$ as
usual.

Applying the Borel construction to $\gamma_{|V|}$ yields a complex
$|V|$-plane bundle, whose Thom space is homeomorphic to
$ET^k_+\wedge_{T^k}\MU_{T^k}(V)$ for every complex representation
$V$. The resulting classifying maps induce maps of pointed homotopy
sets $[S^V\wedge X_+,\MU_{T^k}(W)]_{T^k} \ra[(ET^k_+\wedge
S^V\wedge X_+)/T^k,\MU(|W|)]$, whose target stabilises to the
non-equivariant cobordism algebra $\varOmega_U^{*+2|V|}((S^V\wedge
ET^k_+\wedge X_+)/{T^k})$. Applying the Thom isomorphism yields
the transformation
\[
\alpha\colon\MU^*_{T^k}(X_+)
\lra\varOmega_U^*((ET^k\times_{T^k}X)_+),
\]
which is multiplicative and preserves Thom classes \cite[Proposition
1.2]{tomd70}.

In the language of \cite{co-ma96}, we may summarise the construction
of $\alpha$ using the homomorphism
$\MU^*_{T^k}(X_+)\ra\MU^*_{T^k}((ET^k\times X)_+)$ induced by the
$T^k$- projection $ET^k\times X\ra X$; since $\MU_{T^k}$ is
\emph{split} and $T^k$ acts freely on $ET^k\times X$, the target may
be replaced by $\varOmega_U^*((ET^k\times_{T^k}X)_+)$. Moreover,
$\alpha$ is an isomorphism whenever $X$ is compact and $T^k$ acts
freely.

L\"offler's \emph{completion theorem} \cite{mayetal96} states that
$\alpha$ is precisely the homomorphism of completion with respect to
the augmentation ideal in $\MU^*_{T^k}(X_+)$.

Our primary purpose is to compute $\alpha$ on geometrically
defined cobordism classes, so it is natural to seek an equivariant
version of Quillen's formulation \cite{quil71} of the complex
cobordism functor. However, his approach relies on normal complex
structures, whereas many of our examples present themselves most
readily in terms of tangential information. In the non-equivariant
situation, the two forms of data are, of course, interchangeable;
but the same does not hold equivariantly. This fact was often
ignored in early literature, and appears only to have been made
explicit in 1995, by Comeza\~na \cite[\S3]{come96}. As we shall
see below, tangential structures may be converted to normal, but
the procedure is not reversible.

Over any smooth compact manifold $X$, we formulate our
constructions in terms of smooth bundles
\smash{$E\stackrel{\pi}{\ra}X$} with compact $d$-dimensional fibre
$F$, following Dold \cite{dold78}.
\begin{definitions}\label{sttacotkbd}
  The bundle $\pi$ is \emph{stably tangentially complex} when the
  bundle $\tau_F(E)$ of tangents along the fibre is equipped with a
  stably complex structure $c_\tau(\pi)$.
  Two such bundles $\pi$ and $\pi'$ are \emph{equivalent} when
  $c_\tau(\pi)$ and $c_\tau(\pi')$ are homotopic upon the addition of
  appropriate trivial summands; and their equivalence classes are
  \emph{cobordant} when there exists a third bundle $\rho$, whose
  boundary \smash{$\partial L\stackrel{\rho}{\ra}X$} may be identified
  with $E_1\sqcup E_2\ra X$ in the standard fashion.

  If $\pi$ is $T^k$-equivariant, then it is stably tangentially
  complex \emph{as a $T^k$-bundle} when $c_\tau(\pi)$ is also
  $T^k$-equivariant. The notions of \emph{equivariant equivalence}
  and \emph{equivariant cobordism} apply to such bundles accordingly.
\end{definitions}

We interpret $c_\tau(\pi)$ as a real isomorphism
\begin{equation}\label{ctaupi}
  \tau_F(E)\oplus\bR^{2l-d}\lra\xi\,,
\end{equation}
for some complex $l$-plane bundle $\xi$ over $E$; so the composition
\begin{equation} \label{rtpi}
  r(t)\colon\xi\stackrel{c_\tau^{-1}(\pi)}{\lllra}\tau_F(E)\oplus
  \bR^{2l-d}\stackrel{d\+ (t\cdot)\oplus I}{\llllra}
  \tau_F(E)\oplus\bR^{2l-d} \stackrel{c_\tau(\pi)}{\lllra}\xi
\end{equation}
is a complex transformation for any $t\in T^k$, where $d(t\cdot)$ is
the differential of the action by $t$. In other words, \eqref{rtpi}
determines a representation $r\colon
T^k\to\Hom_\mathbb{C}(\xi,\xi)$.

If $X_+=S^0$, then $X$ is the one point-space $*$, and we may
identify both $F$ and $E$ with some $d$-dimensional smooth
$T^k$-manifold $M$, and $\tau_F(E)$ with its tangent bundle
$\tau(M)$. So $c_\tau(\pi)$ reduces to a stably tangentially
complex $T^k$-equivariant structure $c_\tau$ on $M$, and its
cobordism class belongs to the geometric bordism group
\smash{$\varOmega_d^{U:T^k}$} of \cite[\S3]{come96}. We therefore
denote the set of equivariant cobordism classes of stably
tangentially complex $T^k$-bundles over $X$ by
$\varOmega^{-d}_{U:T^k}(X_+)$. It is an abelian group under
disjoint union of total spaces, and
\smash{$\varOmega^*_{U:T^k}(X_+)$} is a graded
\smash{$\varOmega_*^{U:T^k}\!$}-module under cartesian product.
Furthermore, \smash{$\varOmega^*_{U:T^k}(-)$} is functorial with
respect to pullback along smooth $T^k$-maps $Y\ra X$.

\begin{proposition}\label{cttcn}
Given any smooth compact $T^k$-manifold $X$, there is a canonical
homomorphism
\[
  \nu\colon\varOmega^{-d}_{U:T^k}(X_+)\lra\MU^{-d}_{T^k}(X_+)
\]
for every $d\geq 0$.
\end{proposition}
\begin{proof}
  Let $\pi$ in \smash{$\varOmega^{-d}_{U:T^k}(X_+)$} denote the
  cobordism class of a stably tangentially complex $T^k$-bundle $F\ra
  E\stackrel{\pi}{\ra}X$. Following the construction of the transfer
  \cite{be-go75}, choose a $T^k$-equivariant embedding $i\colon E\ra
  V$ into a unitary $T^k$-representation space $V$ \cite{bred72}, and
  consider the embedding $(\pi,i)\colon E\ra X\times V$; it is
  $T^k$-equivariant with respect to the diagonal action on $X\times
  V$, and its normal bundle admits an equivariant isomorphism
  $c\colon\tau_F(E)\oplus\nu(\pi,i)\ra V$ of bundles over $E$. Now
  combine $c$ with $c_\tau(\pi)$ of \eqref{ctaupi} to obtain an
  equivariant isomorphism
\begin{equation}\label{stnocplxtk}
W\oplus\nu(\pi,i)\lra\xi^{\perp}\oplus V\oplus\bR^{2l-d},
\end{equation}
where $W\letbe \xi^\perp\oplus\xi$ is a unitary $T^k$-decomposition
for some representation space $W$.

If $d$ is even, \eqref{stnocplxtk} determines a complex
$T^k$-structure on an equivariant stabilisation of $\nu(\pi,i)$; if
$d$ is odd, a further summand $\bR$ must be added. For notational
convenience, assume the former, and write $d=2n$. Then Thomify the
classifying map for $\xi^{\perp}\oplus V\oplus\bC^{l-n}$ to give a
$T^k$-equivariant map
\[
S^W\wedge M(\nu(\pi,i))\lra M(\gamma_{|V|+|W|-n}),
\]
and compose with the Pontryagin-Thom construction on $\nu(\pi,i)$ to
obtain
\[
  f(\pi)\colon S^{V\oplus W}\wedge X_+\lra M(\gamma_{|V|+|W|-n}).
\]
If $\pi$ and $\pi'$ are equivalent, then $f(\pi)$ and $f(\pi')$
differ only by suspension; if they are cobordant, then $f(\pi)$
and $f(\pi')$ are stably $T^k$-homotopic. So define $\nu(\pi)$ to
be the $T^k$-homotopy class of $f(\pi)$, as an element of
\smash{$\MU^{-d}_{T^k}(X_+)$}.

The linearity of $\nu$ follows immediately from the fact that addition
in $\varOmega^{-d}_{U:T^k}(X_+)$ is induced by disjoint union.
\end{proof}

The proof of Proposition \ref{cttcn} also shows that $\nu$ factors
through the geometric cobordism group of stably normally complex
$T^k$-manifolds over $X$.
\begin{definition}
For any smooth compact $T^k$-manifold $X$, the \emph{universal toric
genus} is the homomorphism
\[
\varPhi_X\colon\varOmega^*_{U:T^k}(X_+)\lra
\varOmega_U^*((ET^k\times_{T^k}X)_+)
\]
given by the composition $\alpha\cdot\nu$.
\end{definition}
Restricting attention to the case $X_+=S^0$ defines
\[
  \varPhi\colon\varOmega_*^{U:T^k}\lra\varOmega^U_*[[u_1,\dots,u_k]],
\]
where the target is isomorphic to $\varOmega^*_U(BT^k_+)$, and
$u_j$ is the cobordism Chern class $c_1^{M\?U}(\zeta_j)$ of the
conjugate Hopf bundle over the $j$th factor of $BT^k$, for $1\leq
j\leq k$. This version of $\varPhi$ appears in \cite{bu-ra07} for
$k=1$, and is essentially the homomorphism introduced
independently by L\"offler \cite{loff74} and Krichever
\cite{kric76}. It is defined on triples $(M,a,c_\tau)$, where $a$
denotes the $T^k$-action; whenever possible, we omit one or both
of $a$, $c_\tau$ from the notation.

We follow Krichever by interpreting $\varPhi$ as an equivariant
genus, in the sense that it is a multiplicative cobordism invariant
of stably complex $T^k$-manifolds, and takes values in a graded ring
with explicit generators and relations. As such it is an equivariant
extension of Hirzebruch's original notion of genus \cite{hirz66},
and is closely related to the theory of formal group laws.

When $X_+=S^0$, Hanke \cite{hank05} and L\"offler \cite[(3.1)
Satz]{loff74} prove that $\nu$ and $\alpha$ are monic; therefore
so is $\varPhi$. On the other hand, there are two important
reasons why $\nu$ cannot be epic. Firstly, it is defined on stably
tangential structures by converting them into stably normal
information; this procedure cannot be reversed equivariantly,
because the former are stabilised only by trivial representations
of $T^k$, whereas the latter are stabilised by arbitrary
representations $V$. Secondly, homotopical cobordism groups are
periodic, and each representation $W$ gives rise to an invertible
\emph{Euler class} $e(W)$ in $\MU^{T^k}_{-|W|}$ whenever the fixed
point set $\fix(W)$ is non-empty; this phenomenon exemplifies the
failure of equivariant transversality. Since tom Dieck also proves
that $\alpha$ is not epic, the same is true of $\varPhi$.

It is now convenient to describe $\varPhi_X$ purely in terms of stably
complex structures, by using geometrical models for the target algebra
$\varOmega_U^*((ET^k\times_{T^k}X)_+)$.

Let $\pi$ in \smash{$\varOmega^{-d}_{U:T^k}(X_+)$} be represented by
$\pi\colon E\ra X$, as in the proof of Proposition \ref{cttcn}, and
consider the smooth fibre bundle
\begin{equation}\label{phitan}
F\lllra ET^k\times_{T^k}E\stackrel{1\times_{\?T^{\?k}}\pi}{\lllra}
ET^k\times_{T^k}X
\end{equation}
obtained by applying the Borel construction. The bundle of
tangents along the fibre is the Borelification
$1\times_{T^k}\tau_F(E)$, and so inherits a stably complex
structure $1\times_{T^k}c_\tau(\pi)$. Moreover,
$1\times_{T^k}c_\tau(\pi)$ and $1\times_{T^k}c_\tau(\pi')$ are
equivalent whenever $\pi$ and $\pi'$ are equivalent
representatives for $\pi$, and cobordant whenever $\pi$ and $\pi'$
are cobordant, albeit over the infinite dimensional manifold
$ET^k\times_{T^k}X$. In this sense, \eqref{phitan} represents
$\varPhi_X(\pi)$. For a more conventional description, we convert
\eqref{phitan} to a Quillen cobordism class by mimicking the proof
of Proposition \ref{cttcn} as follows.

The Borelification $h$ of the embedding $(\pi,i)$ is also an
embedding, whose normal bundle admits the stably complex structure
that is complementary to $1\times_{T^k}c_\tau(\pi)$. So the
factorisation
\begin{equation}\label{phinorm}
ET^k\times_{T^k}E\stackrel{h}{\llra}ET^k\times_{T^k}(X\times V)
\stackrel{r}{\llra}ET^k\times_{T^k}X
\end{equation}
determines a \emph{complex orientation} for $1\times_{T^k}\pi$, as
defined in \cite[(1.1)]{quil71}. The map $1\times_{T^k}\pi$ has
dimension $d$, and its construction preserves cobordism classes in
the appropriate sense; applying the Pontryagin-Thom construction
confirms that it represents $\varPhi_X(\pi)$ in
$\varOmega_U^{-d}((ET^k\times_{T^k}X)_+)$. Following
\cite[(1.4)]{quil71}, we may therefore express $\varPhi_X(\pi)$ as
\smash{$(1\times_{T^k}\pi)_*1$}, where
\[
  (1\times_{T^k}\pi)_*\colon\varOmega^*_U(ET^k\times_{T^k}E)\llra
  \varOmega^{*-d}_U(ET^k\times_{T^k}X),
\]
denotes the \emph{Gysin homomorphism}. For notational simplicity, we
abbreviate $\varPhi_X(\pi)$ to $\varPhi(\pi)$ from this point onward.

In all relevant examples, the infinite dimensionality of
$ET^k\times_{T^k}X$ presents no problem because $ET^k$ may be
approximated by the compact manifolds $\varPi\?S(q)$ of
\eqref{prodsphs}, and $\lim^1$ arguments applied. In certain
circumstances, infinite dimensional manifolds may actually be
incorporated into the definitions, as proposed in \cite{ba-oz00}.

To be more explicit, we use the model $(\bC P^\infty)^k$ for
$BT^k$. The first Chern class in $\varOmega^2_U(\bC P^q)$ is
represented geometrically by an inclusion $\bC P^{q-1}\ra\bC P^q$,
whose normal bundle is $\zeta$. As $q$ increases, these classes
form an inverse system, whose limit defines $u\letbe x^{MU}$ in
$\varOmega^2_U(\bC P^\infty)$; it is represented geometrically by
the inclusion of a hyperplane. An additive basis for
$\varOmega_*^U[[u_1,\dots,u_k]]$ in dimension $2|\omega|$ is given
by the monomials $u^\omega=u_1^{\omega_1}\cdots u_k^{\omega_k}$,
where $\omega$ ranges over nonnegative integral vectors
$(\omega_1,\ldots,\omega_k)$, and
\smash{$|\omega|=\sum_j\omega_j$}. Every such monomial is
represented geometrically by a $k$-fold product of complex
subspaces of codimension $(\omega_1,\ldots,\omega_k)$ in $(\bC
P^\infty)^k$, with normal bundle
$\omega_1\zeta_1\oplus\dots\oplus\omega_k\zeta_k$.

Now restrict attention to the case $X=*$. Given the equivariant
cobordism class of $(M^{2n},a,c_\tau)$, we may approximate
\eqref{phitan} and \eqref{phinorm} over $\varPi\?P(q)$ by stably
tangential and stably normal complex structures on the
$2(kq+n)$-dimensional manifold $W_q=\varPi\?S(q)\times_{T^k}M$.
The corresponding complex orientation is described by
\begin{equation}\label{phim}
  W_q\stackrel{i}{\lra}\varPi\?S(q)\times_{T^k}V
  \stackrel{r}{\lra}\varPi\?P(q),
\end{equation}
where $i$ denotes the Borelification of a $T^k$-equivariant
embedding $M\ra V$, and $r$ is the complex vector bundle induced
by projection. The normal bundle $\nu(i)$ is invested with the
complex structure complementary to $1\times_{T^k}c_\tau$, which
determines a complex cobordism class in $U^{-2n}(\varPi\? P(q))$.
As $q$ increases, these classes form an inverse system, whose
limit is $\varPhi(M,c_\tau)$ in $U^{-2n}(BT^k)$; it is represented
geometrically by the complex orientation $ET^k\times_{T^k}M\ra
ET^k\times_{T^k}V\ra BT^k$, which factorises the projection
\smash{$1\times_{T^k}\pi$}. In particular, $\varPhi(M,c_\tau)$ is
the Gysin image \smash{$(1\times_{T^k}\pi)_*1$}.

If we write
\begin{equation}\label{defgomega}
  \varPhi(M,c_\tau)\;=\;\sum_\omega g_\omega(M)\,u^\omega
\end{equation}
in $\varOmega_*^U[[u_1,\dots,u_k]]$, then the coefficients
$g_\omega(M)$ lie in $\varOmega^U_{2(|\omega|+n)}$, and their
representatives may be interpreted as universal operations on the
cobordism class of $M$. If $c_\tau$ is converted to stably
\emph{normal} data, then L\"offler \cite[(3.2) Satz]{loff74} has
described these operations using constructions of Conner and Floyd and
tom Dieck, as follows.

Let $S^3\subset\bC^2$ be a $T$-space, via
$t\cdot(z_1,z_2)=(tz_1,t^{-1}z_2)$, and let $T$ operate on $M$ by
restricting the $T^k$-action to the $j$th coordinate circle, for
$1\leq j\leq k$. So $\varGamma_j(M)\letbe M\times_TS^3$ is a
stably normally complex $T^k$-manifold under the action
\[
(t_1,\dots,t_k)\cdot[m,(z_1,z_2)]\;=\;
[(t_1,\dots,t_k)\cdot m,(t_jz_1,z_2)],
\]
and the operations $\varGamma_j$ may be composed.
\begin{proposition}\label{loffform}
  The coefficient $g_\omega(M)$ of \eqref{defgomega} is represented by
  the stably complex manifold
  $\varGamma^\omega(M)\letbe
  \varGamma_1^{\omega_1}\dots\varGamma_k^{\omega_k}(M)$;
  in particular, $g_0(M)=[M]$ in $\varOmega^U_{2n}$.
\end{proposition}
We shall rewrite L\"offler's formula in terms of stably tangential
data in \S\ref{geri}, and describe its application to rigidity
phenomena.

There is an elegant alternative expression for $\varPhi(M,c_\tau)$
when the fixed points $x$ are isolated, involving their \emph{weight
vectors} $w_j(x)\in\bZ^k$, for $1\leq j\leq n$. Each such vector
determines a line bundle
\[
  \zeta^{w_j(x)}\;\letbe\;
  \zeta_1^{w_{j,1}(x)}\otimes\dots\otimes\zeta_k^{w_{j,k}(x)}
\]
over $BT^k$, whose first Chern class is a formal power series
\begin{equation}\label{powersys}
[w_j(x)](u)\;\letbe\; \sum_\omega
a_\omega[w_{j,1}(x)](u_1)\,^{\omega_1}\,\cdots\,
[w_{j,k}(x)](u_k)\,^{\omega_k}
\end{equation}
in $\varOmega_U^2(BT^k)$. Here $[m](u_j)$ denotes the power series
$c_1^{M\?U}\!(\zeta_j^m)$ in $\varOmega_U^2(\bC P^\infty)$ for any
integer $m$, and the $a_\omega$ are the coefficients of
$c_1^{M\?U}\!(\zeta_1\otimes\dots\otimes\zeta_k)$. The $[m](u_j)$
form the \emph{power system} or \emph{$m$-series} of the universal
formal group law $F(u_1,u_2)$, and the $a_\omega$ are the
coefficients of its iteration in $\varOmega^U_{2|\omega|}$; see
\cite{b-m-n71}, \cite{bu-no71} for further details. Modulo
decomposables we have that
\begin{equation}\label{wjxmoddec}
[w_j(x)](u_1,\dots,u_k)\;\equiv\;w_{j,1}u_1+\dots +w_{j,k}u_k,
\end{equation}
and it is convenient to rewrite the right hand side as a scalar
product $w_j(x)\cdot u$.

Krichever \cite[(2.7)]{kric74} obtained the following localisation
formula.
\begin{proposition}\label{krich}
  If the structure $c_\tau$ is almost complex and the set of fixed
  points $x\in M$ is finite, then the equation
\begin{equation}\label{kf}
  \varPhi(M,c_\tau)\;=\;\sum_{F\?i\+x(M)}\;
  \prod_{j=1}^n \frac1{[w_j(x)](u)}
\end{equation}
is satisfied in $\varOmega_U^{-2n}(BT^k)$.
\end{proposition}
All terms of negative degree must cancel in the right-hand side of
\eqref{kf}, imposing strong restrictions on the normal data of the
fixed point set; these are analogues of Novikov's original
Conner-Floyd relations. We shall extend Krichever's formula to stably
complex structures, and discuss applications to non-equivariant
cobordism.

%
%
%
%
%
%
%
%
%

\section{Genera and rigidity}\label{geri}

In this section we consider multiplicative cobordism invariants of
stably complex manifolds, and define equivariant extensions to
\smash{$\varOmega_*^{U:T^k}$} using $\varPhi$. We recover a
generalisation of L\"offler's formula from the tangential
viewpoint, explain its relationship with bounded flag manifolds,
and discuss consequences for a universal concept of rigidity. We
draw inspiration from Hirzebruch's original theory of genera, but
give equal weight to Novikov's interpretation in terms of formal
group laws and Adams's translation into complex-oriented
cohomology theory.


\subsection{Background and examples}
Hirzebruch's \cite{hirz66} studies formal power series over a
commutative ring $R$ with identity. For every $Q(x)\equiv 1$ mod $x$
in $R[[x]]$, he defines a multiplicative homomorphism
$\ell_Q\colon\varOmega_*^U\ra R$, otherwise known as the genus
associated to $Q$. The image of $\ell_Q$ is a graded subring
$R'_*\leq R$, whose precise description constitutes an
\emph{integrality
  theorem} for $\ell_Q$. Hirzebruch also shows that his procedure is
reversible. For every homomorphism $\ell\colon\varOmega^U_*\ra R$ into
a torsion-free ring, he defines a formal power series $Q_\ell(x)\equiv
1$ mod $x$ in $R\otimes\bQ[[x]]$. In this context, the integrality
theorem corresponds to describing the subring $R''\subseteq
R\otimes\bQ$ generated by the coefficients of $Q_\ell$.

The interpretation of genera in terms of formal group laws $F$ was
pioneered by Novikov \cite{novi67}, \cite{novi68}, and is surveyed in
\cite{b-m-n71}. His crucial insight identifies the formal power series
$x/Q(x)$ with the exponential series $f_F(x)$ over $R\otimes\bQ$, and
leads directly to the analytic viewpoint of $f_F(x)$ as a function
over $R\otimes\bC$. The associated genus $\ell_F$ actually classifies
$F$, because $\varOmega_*^U$ may be identified with the Lazard ring,
following Quillen \cite{quil69}. These ideas underlie many subsequent
results of the Moscow school, including generalisations of the
classical elliptic genus by Krichever \cite{kric90} and Buchstaber
\cite{buch90}.

Our third, and least familiar, approach to genera involves complex
oriented spectra $D$ and $E$, and may initially appear more
complicated. Nevertheless, it encodes additional homotopy theoretic
information, as was recognised by developers such as Adams
\cite{adam74} and Boardman \cite{boar67}. Many cobordism related
calculations in western literature have been influenced by their point
of view, including those of \cite{bu-ra98} and \cite{ray72}. We use
the notation of Section \ref{in}, especially Example \ref{hmuexa}.
\begin{definition}\label{defhtpclgenus}
  A \emph{homotopical genus} $(t^D\!,v)$ with \emph{defining spectrum $D$}
  and \emph{evaluating spectrum} $E$ consists of a Thom class $t^D\!$,
  and homotopy commutative diagram
\[
\begin{CD}
  \MU@>t^D>>D\\
  @Vh^{M\!\?U}\!VV@VVvV\\
  H\wedge\MU@>s^E>>E
\end{CD}
\]
of complex orientable spectra, where $D_*$ and $E_*$ are
concentrated in even dimensions and free of additive torsion, and
$v_*\colon D_*\ra E_*$ is monic. The \emph{associated
orientations} are the classes $s^E_*x^H$ and $v_*x^D$ in $E^2(\bC
P^\infty)$, abbreviated to $x^H$ and $x^E$ respectively.
\end{definition}

It follows from Definition \eqref{defhtpclgenus} that $x^E$
extends to the Thom class $t^E\letbe v_*t^D$. For fixed $t^D$, the
\emph{initial} $v$ is $h^D\colon D\ra H\wedge D$, with $s^{H\wedge
D}=1\wedge t^D$, and the \emph{final} $v$ is rationalisation
$r_D\colon D\ra D\bQ$, with $s^{D\bQ}=r_H\wedge t^D$. If $D$ is an
$H$-module spectrum or is rational, then we may take $v$ to be the
identity $1^D$.

The formal group law $F=F_{(t^D,v)}$ associated to a homotopical
genus is classified by the homomorphism $t^D_*\colon\varOmega^U_*\ra
D_*$ of graded rings. Its exponential and logarithm are the formal
power series $b^E(x^H)$ and $m^E(x^E)$ of \eqref{bmseries}, which
link the associated orientations in $E^2(\bC P^\infty)$; then
$F(u_1,u_2)$ is given by $b^E(m^E(u_1)+m^E(u_2))$ over $E_*$. The
bordism classes of the stably complex Milnor hypersurfaces
$H^j(1)\subset\times^{j+1}S^2$ are rational polynomial generators
for $\varOmega^U_*$, as are those of the projective spaces $\bC
P^j$, for $j\geq 0$. So the homomorphism $t^E_*$ is evaluated by
either of the sequences
\begin{equation}\label{hteonhorcp}
  t^E_*(H^j(1))\;=\;(j+1)!\,b^E_j\sors t^E_*(\bC P^j)\;=\;(j+1)\,m^E_j
\end{equation}
in $E_*$, following \cite[(2.11)]{ray87} and
\cite[Appendix~1]{novi67} respectively. Since $v_*$ is monic,
$t^D_*$ may be recovered from \eqref{hteonhorcp}.

Not all Hirzebruch genera are homotopical, but the following procedure
usually creates a workable link. For any ring homomorphism
$\ell\colon\varOmega_*^U\ra R$, first identify a complex oriented
spectrum $D$ such that some natural homomorphism $D_*\ra R$
approximates the image of $\ell$, and the corresponding Thom class
$t^D\colon\MU\ra D$ induces $\ell$ on homotopy rings. Then identify a
map $v\colon D\ra E$ of complex oriented ring spectra such that
$E_*\leq D_*\otimes\bQ$ approximates the graded subring generated by
the coefficients of the exponential series $f_\ell(x)$, and
$t^E=v\cdot t^D$ factors through $h^{M\?U}$. Should $D_*$ already be
isomorphic to an appropriately graded version of $R$, we indicate the
forgetful homomorphism by $R_*\ra R$.

The existence of $E$ amounts to a homotopical integrality
theorem, and $f_\ell(x^H)$ is synonymous with $b^E(x^H)$ in
$E_*[[x^H]]$. The initial example $(t^D,h^D)$ sheds light on the
integrality properties of arbitrary Hirzebruch genera, homotopical
or not, by suggesting the option of linearising the formal group law
$F_\ell$ over the extension
\begin{equation}\label{niceext}
R\;\leq\;R\otimes_{\varOmega_*^U}H_*(\MU)\;\leq\;R\otimes\bQ\,.
\end{equation}
An alternative implementation of \eqref{niceext} is to work
exclusively with the ring of \emph{Hurwitz series} over $R$
\cite{car49}, following the lead of umbral calculus \cite{ray87}. From
the analytic and formal group theoretic viewpoints, $F_\ell$ is
usually linearised over $R\otimes\bQ$, and homotopy theoretic
information is then lost.

We emphasise that the art of calculation with homotopical genera may
still rely upon the Hirzebruch's original methods, as well as
Krichever's analytic approach. By way of illustration, we rederive
Hirzebruch's famous formula \eqref{ht2eq} for computing $t^E_*$ on an
arbitrary stably complex manifold $(M^{2n},c_\tau)$. Sometimes, it
proves more convenient to work with the complementary normal structure
$c_\nu$, which arises naturally from the Pontryagin-Thom
construction. The virtual bundles $\tau$ and $\nu$ are related by
$\perp^*\!\nu=\tau$, where $\perp\colon\BU\ra\BU$ is the involution
that classifies the complement of the universal virtual bundle. It
induces an automorphism $\perp^*$ of $H\wedge\MU$, which is specified
uniquely by its action on $(H\wedge\MU)^2(\bC P^\infty)$; we write
$\perp^*x^{MU}$ as $a(x^H)\letbe x^\perp$, and note that it is
determined by either of the identities
\begin{equation}\label{xperp}
a(x^H)\;=\;(x^H)^2/b(x^H)\sporsp a_+(x^H)\;=\;1/b_+(x^H).
\end{equation}
So $x^\perp$ is also a complex orientation, with $a_1=-b_1$ and
$a_2=b_1^2-b_2$, for example.

The complex orientability of $M$ ensures the existence of
fundamental classes $\sigma^H_M$ and $\sigma^E_M$ in $E_{2n}(M)$.
The coefficients $b_j\letbe b^E_j$ of $b^E(x^H)$ generate monomials
$b^\omega\letbe b_1^{\omega_1}\dots b_n^{\omega_n}$ in
$E_{2|\omega|}$, which may not be independent. On the other hand,
the orientation class $x^H$ defines basis elements $c^H_\omega$ in
$E^{2|\omega|}(BU)$, of which $c^H_j$ restricts to the
$j\,$th Chern class in $H^{2j}(BU)$, for every $j\geq 0$.
\begin{proposition}\label{ht1form}
The homomorphism $t^E_*$ satisfies
\begin{equation}\label{ht1eq}
t^E_*(M,c_\nu)\;=\;\sum_\omega\big\langle b^\omega
c^H_\omega(\nu)\,,\, \sigma^H_M\big\rangle
\end{equation}
in $E_{2n}$, for any stably complex manifold $(M,c_\nu)$.
\end{proposition}
\begin{proof}
The case $(1^{M\?U},h^{M\?U})$ is universal. In $(H\wedge\MU)_{2n}$,
the formula
\[
  h^{M\?U}_*(M,c_\nu)\;=\;\sum_\omega\big\langle b^\omega
  c^H_\omega(\nu)\,,\sigma^H_M\big\rangle
\]
holds because $b^\omega$ and $c^H_\omega$ define dual bases; so
\eqref{ht1eq} follows by applying $s^E_*$.
\end{proof}
\begin{corollary}\label{hurhirz}
The value of $t^E_*(M,c_\nu)$ may be rewritten as
\begin{equation}\label{ht2eq}
\bigg\langle\prod_i b_+(x^H_i)(\nu)\,, \sigma^H_M\bigg\rangle\;=\;
\bigg\langle\prod_i a_+(x^H_i)(\tau)\,, \sigma^H_M\bigg\rangle.
\end{equation}
\end{corollary}
\begin{proof}
These formulae are to be interpreted by writing $c^H_\omega$ as the
appropriate symmetric function in the variables $x^H_1$, $x^H_2$,
\dots. So
\[
\prod_i\big(1+b_1x^H_i+b_2(x^H_i)^2+\dots\big)\;=\;\sum_\omega
b^\omega c^H_\omega,
\]
and \eqref{ht2eq} follows immediately from \eqref{ht1eq} and
\eqref{xperp}.
\end{proof}

We now describe six homotopical genera $(t^D,v)$, and evaluate
each $t^E_*$ explicitly; (1), (2), (3) and (6) are well-known, but
(4) and (5) may be less familiar.

\begin{examples}\label{cgenexs}\hfill

(1) The \emph{augmentation genus} $ag$ and \emph{universal
genus} $ug$ are the extreme cases, given by $(t^H,1^H)$ and
$(1^{M\?U},h^{M\?U})$ respectively. Thus $ag(M)=0$ and $ug(M)=[M]$
for any stably complex $M$ of dimension $>0$.

(2) The \emph{Hurewicz genus} $hr$ of Example \ref{hmuexa} is given
by $(h^{M\?U},1^{H\wedge M\?U})$, where
$(H\wedge\MU)_*\cong\bZ[b_1,b_2,\dots]$. Then
\[
x^{M\?U}\;=\;b(x^H)\sands x^H\;=\;m(x^{M\?U})
\]
in $(H\wedge\MU)^2(\bC P^\infty)$, so $hr(H^j(1))=(j+1)!\,b_j$ and
$hr(\bC P^j)=(j+1)m_j$ for $j\geq 0$.

(3) The \emph{Todd genus} $td$ of \cite{co-fl66} is given by
$(t^K,h^K)$ , where $K_*\cong\bZ[z,z^{-1}]$, and $(H\wedge
K)_*\cong\bQ[z,z^{-1}]$. Here
\[
x^K\;=\;(\exp zx^H-1)/z\sands x^H\;=\;(\ln(1+zx^K))/z
\]
in $(H\wedge K)^2(\bC P^\infty)$, so $td(H^j(1))=z^j$ and $td(\bC
P^j)=(-z)^j$.

(4) The \emph {$c_n$ genus} $cg$ is given by $(t^C,1^C)$, where $C$
denotes $H\wedge X(2)$ for the Thom spectrum $X(2)$ of \cite[\S
9.1]{rave92} over $\varOmega S^3$; thus $C_*\cong\bZ[v]$. Here
\[
x^C\;=\;x^H/(1+vx^H)\sands x^H\;=\;x^C/(1-vx^C)
\]
in $C^2(\bC P^\infty)$, so $cg(H^j(1))=(j+1)!\,(-v)^j$ and
$cg(\bC P^j)=(j+1)v^j$.

(5) The \emph{Abel genus} $ab$ and \emph{2-parameter Todd genus}
\emph{t2} are given by $(t^A,1^L)$ and $(t^{K\?2},1^L)$ respectively,
where $L$ denotes $H\wedge K\wedge K$; thus $L_*\cong\bQ[y^{\pm
  1},z^{\pm 1}]$. Here
\[
x^A\!=(\exp yx^H\!\!-\exp zx^H)/(y-z)\quad\text{and}\quad
x^{K\?2}\!=(y-z)x^A/(y\exp zx^H\!\!-z\exp yx^H)
\]
in $L^2(\bC P^\infty)$, so $ab(H^j(1))=\sum_{i=0}^{j}y^iz^{j-i}$
and \emph{t2}$(H^j(1))=\sum_\rho y^{r(\rho)}z^{f(\rho)}$, where
each permutation $\rho\in S_{j+1}$ has $r(\rho)$ rises and
$f(\rho)$ falls \cite[\S2.4.21]{go-ja86}; similarly,
\emph{t2}$(\bC P^j)=\sum_{i=0}^j(-1)^jy^iz^{j-i}$. If $y\?=\?0$
then $ab$, $\mathit{t2}$, and $td$ agree, and if $y\?=\?z\?=\?-v$
then $\mathit{t2}$ is $cg$; if $y\?=\?z$ then $x^A\?=\?x^H\exp
yx^H$. The Abel genus appears in \cite{bu-kh90}.

(6) The \emph{signature} $sg$ is given by $(t^S,h^{K\?O[1/2]})$, where
$t^S\colon\MU\ra\KO[1/2]$ is the Thom class of \cite{hirz66},
$\KO[1/2]_*\cong\bZ[1/2][z^{\pm2}]$, and
$(H\wedge\KO[1/2])_*\cong\bQ[z^{\pm 2}]$. Then
\[
x^S\;=\;(\tanh zx^H)/z\sands x^H=(\tanh^{-1}zx^H)/z
\]
in $(H\wedge\KO[1/2])^2(\bC P^\infty)$, so $sg(\bC P^{2j})=z^{2j}$.
In fact $sg$ is $t2$ in case $y=-z$.
\end{examples}
\begin{remarks}\label{normtotan}\hfill\\
  \indent (1) Alternative formulae arise by composing with $\perp^*$
  of \eqref{xperp}, as in Corollary \ref{hurhirz}. The universal such
  example is the \emph{tangential genus} $tg$, induced by the Thom class
  $t^\perp\colon\MU\ra H\wedge\smash\MU$ and determined by
  $tg(H^j(1))=(j+1)!\,a_j$.

  (2) The $c_n$ genus and signature are given globally by
  $cg(M^{2n},c_\tau)\?=\?\langle c^H_n(\tau),\sigma^H_M\rangle v^n$ and
  $sg(M^{2n},c_\tau)=sg(M)z^n$. If $c_\tau$ is almost complex,
  then $cg$ takes the value $\chi(M)v^n$, and is independent of the
  structure; $sg$ is independent of $c_\tau$ in all cases.

  (3) The genus $ab$ may also be described homotopically as
  $(t^{Ab},h^{Ab})$, where $Ab$ denotes the complex oriented theory
  constructed in~\cite{busa02}.
\end{remarks}


\subsection{Equivariant extensions}
Every genus $\ell\colon\varOmega_*^U\ra R_*$ has a
\emph{$T^k$-equivariant extension}
\begin{equation}\label{equivtgenus}
\ell^{\,T^k}\colon\varOmega_*^{U:T^k}\lra R_*[[u_1,\dots,u_k]],
\end{equation}
defined as the composition $\ell\cdot\varPhi$. In this
context, \eqref{defgomega} yields the expression
\begin{equation}\label{ltkexpress}
\ell^{\,T^k}(M,c_\tau)\;=\;\ell(M)\;+
\sum_{|\omega|>0}\ell(g_\omega(M))\,u^\omega.
\end{equation}
In particular, the $T^k$-equivariant extension of the universal
Examples \ref{cgenexs}(1) is $\varPhi$; hence the name \emph{universal
toric genus}.

More generally, we consider elements $\pi$ of
$\varOmega^{-2n}_{U:T^k}(X_+)$, but restrict attention to those $X$
for which $\varOmega^*_U(ET^k\times_{T^k}X)$ is a finitely generated
free $\varOmega^*_U(BT^k)$-module. Then there exist generators
$v_0$, $v_1$, \dots $v_p$ in $\varOmega^*_U(ET^k\times_{T^k}X)$,
where $v_0=1$ and $\dim v_j=2d_j$, which restrict to a basis for
$\varOmega_U^*(X)$ over $\varOmega_*^U$. The elements $v_ju^\omega$
form an $\varOmega_*^U$-basis for
$\varOmega^*_U(ET^k\times_{T^k}X)$; we write their duals as
$a_{j,\omega}$, and choose singular stably complex manifolds
$f_{j,\omega}\colon A_{j,\omega}\ra ET^k\times_{T^k}X$ as their
representatives. We may then generalise \eqref{defgomega} to
\begin{equation}\label{defgpsi}
\varPhi(\pi)\;=\;\sum_{j,\omega}g_{j,\omega}(\pi)\,v_ju^\omega
\end{equation}
in $\varOmega^{-2n}_U((ET^k\times_{T^k}X)_+)$, where the
coefficients $g_{j,\omega}(\pi)$ lie in
$\varOmega^U_{2(n+d_j+|\omega|)}$. Moreover,
$\sum_jg_{j,0}(\pi)\,v_j$ represents the non-equivariant cobordism
class of the complex oriented map $\pi$ in $\varOmega_U^{-2n}(X_+)$.

The extension
\begin{equation}\label{equivtgenusx}
\ell^{\,T^k}(\pi)\colon\varOmega_*^{U:T^k}(X_+)\llra
R_*[[u_1,\dots,u_k]]\langle v_0,v_1,\dots,v_p\rangle
\end{equation}
is the composition $\ell\cdot\varPhi(\pi)$, and is evaluated by
applying $\ell$ to each $g_{j,\omega}(\pi)$. Of course,
\eqref{equivtgenusx} reduces to \eqref{equivtgenus} when $X=*$ and
$p=0$. Alternatively, we may focus on homotopical genera $t^E_*$,
and define the extension $(t^E_X)_*$ as the composition
$t^E_*\cdot\varPhi(\pi)\colon\varOmega^*_{U:T^k}(X_+)\ra
E^*((ET^k\times_{T^k}X)_+)$ for \emph{any} compact $T^k$-manifold
$X$.

As an aid to evaluating equivariant extensions, we identify the
$g_{j,\omega}(\pi)$ of \eqref{defgpsi}.
\begin{theorem}\label{Gpirepsgpi}
For any smooth compact $T^k$-manifold $X$ as above, the complex
cobordism class $g_{j,\omega}(\pi)$ is represented on the
$2(n+d_j+|\omega|)$-dimensional total space of a fibre bundle $F\ra
G_{j,\omega}(\pi)\ra A_{j,\omega}$, whose stably complex structure
is induced by those on $\tau(A_{j,\omega})$ and $\tau_F(E)$; in
particular, $g_{0,0}(\pi)=[F]$ in $\varOmega^U_{2n}$.
\end{theorem}
\begin{proof}
Formula \eqref{defgpsi} identifies $g_{j,\omega}(\pi)$ as the
Kronecker
  product $\langle\varPhi(\pi)\,,a_{j,\omega}\rangle$. In terms of
  \eqref{phim}, it is represented on the pullback of the diagram
\[
  A_{j,\omega}\stackrel{f_{j,\omega}}{\lllra}\varPi\?S(q)\times_{T^k}X
\stackrel{1\times_{\?T^{\?k}}\pi}{\lllla}\varPi\?S(q)\times_{T^k}E
\]
for suitably large $q$, and therefore on the pullback of the diagram
\[
A_{j,\omega}\stackrel{f_{j,\omega}}{\lllra}ET^k\times_{T^k}X
\stackrel{1\times_{\?T^{\?k}}\pi}{\lllla}ET^k\times_{T^k}E.
\]
of colimits. So the representing manifold is the total space of the
pullback bundle $F\ra
G_{j,\omega}(\pi)\stackrel{p}{\ra}A_{j,\omega}$. Moreover, both
$\tau_F(E)$ and $\tau(A_{j,\omega})$ have stably complex structures,
so the structure on the pullback is induced by the canonical
isomorphism
\begin{equation}\label{gpsipistruc}
\tau(G_{j,\omega}(\pi))\;\cong\;\tau_F(G_{j,\omega}(\pi))\oplus
p^*\tau(A_{j,\omega})
\end{equation}
of real vector bundles.
\end{proof}

In order to work with the case $X=*$, we must choose geometric
representatives for the basis elements $b_\omega\letbe
b_{\omega_1}\otimes\dots\otimes b_{\omega_k}$ of
$\varOmega^U_*(BT^k)$, dual to the monomials $u^\omega$ in
$\varOmega^*_U(BT^k)$. For this purpose, we recall the \emph{bounded
flag manifolds} $B_j$.

The subspace $(S^3)^j\subset\bC^{2j}$ consists of all vectors
satisfying $|z_i|^2+|z_{i+j}|^2=1$ for $1\le i\le j$, and is acted
on freely by $T^j$ according to
\begin{equation}\label{tjbflag}
  t\cdot(z_1,\ldots,z_{2j})\;=\;
  (t_1^{-1}z_1,t_1^{-1}t_2^{-1}z_2,\ldots,t_{j-1}^{-1}t_j^{-1}z_j,
  t_1z_{j+1},\ldots,t_jz_{2j})
\end{equation}
for all $t=(t_1,\dots,t_j)$. The quotient manifold
$B_j\letbe(S^3)^j/T^j$ is a $j$-fold iterated $2$-sphere bundle over
$B_0=*$, and for $1\le i\le j$ admits complex line bundles
\[
  \psi_i\colon (S^3)^j \times_{T^j}\mathbb{C}\lra B_j
\]
via the action $t\cdot z = t_iz$ for $z\in\mathbb{C}$. The $B_j$
are called bounded flag manifolds in \cite{bu-ra98}, because their
points may be represented by certain complete flags in
$\bC^{j+1}$.

For any $j>0$, \eqref{tjbflag} determines an explicit isomorphism
\[
  \tau(B_j)\oplus\mathbb{C}^j
  \;\cong\;\psi_1\oplus\psi_1\psi_2\oplus\dots\oplus\psi_{j-1}\psi_j
  \oplus\bar{\psi}_1\oplus\dots\oplus\bar{\psi}_j
\]
of real $4j$-plane bundles, which defines a stably complex
structure $c^\partial_j$ on $B_j$. This extends over the
associated $3$-disc bundle, and so ensures that
$[B_j,c^\partial_j]=0$ in $\varOmega_{2j}^U$. The $B_j$ may also
be interpreted as \emph{Bott towers} in the sense of
\cite{gr-ka94}, and therefore as complex algebraic varieties;
however, the corresponding stably complex structures have nonzero
Chern numbers, and are inequivalent to $c^\partial_j$. We write
the cartesian product $B_{\omega_1}\times\dots\times B_{\omega_k}$
as $B_\omega$, with the bounding stably complex structure
\smash{$c^\partial_\omega$}.

\begin{proposition}\label{udual}
The basis element $b_\omega\in\varOmega^U_{2|\omega|}(BT^k)$ is
represented geometrically by the classifying map
\[
  \psi_\omega\colon B_\omega\longrightarrow BT^k
\]
for the external product
$\psi_{\omega_1}\!\times\dots\times\psi_{\omega_k}$ of circle
bundles.
\end{proposition}
\begin{proof}
  This follows directly from the case $k=1$, proven in
  \cite[Proposition 2.2]{ray86}.
\end{proof}

We may now return to Proposition \ref{loffform}, and the classes
$g_\omega(M)$ of \eqref{defgomega}.

\begin{definitions}\label{Gomegadef}
  Let $T^\omega$ be the torus $T^{\omega_1}\times\ldots\times
  T^{\omega_k}$ and $(S^3)^\omega$ the product
  $(S^3)^{\omega_1}\times\ldots\times (S^3)^{\omega_k}$, on which
  $T^\omega$ acts coordinatewise. The manifold $G_\omega(M)$ is the
  orbit space $(S^3)^\omega\times_{T^{\omega}}M$, where $T^\omega$
  acts on $M$ via the representation
\begin{equation}\label{tomegatk}
(t_{1,1},\ldots,t_{1,\omega_1};\,\dots\,;t_{k,1},\ldots,t_{k,\omega_k})
\;\longmapsto\;(t^{-1}_{1,\omega_1},\dots,t^{-1}_{k,\omega_k}).
\end{equation}
The stably complex structure $c_\omega$ on $G_\omega(M)$ is induced
by the tangential structures $c_\tau$ and $c^\partial_\omega$ on the
base and fibre of the bundle $M\ra G_\omega(M)\ra B_\omega$.
\end{definitions}
\begin{corollary}\label{Gmrepsgm}
  The manifold $G_\omega(M)$ represents $g_\omega(M)$ in
  $\varOmega^U_{2(|\omega|+n)}$.
\end{corollary}
\begin{proof}
  Apply Theorem \ref{Gpirepsgpi} in the case $X=*$.  Then $\pi$
  reduces to the identity map on $M$, and $G_{0,\omega}(\pi)$ reduces to
  $G_\omega(M)$; furthermore the stably complex structures
  \eqref{gpsipistruc} and $c_\omega$ of Definition \ref{Gomegadef} are
  equivalent.
\end{proof}
Note that L\"offler's $\varGamma^\omega(M)$ is $T^k$-equivariantly
diffeomorphic to $G_\omega(M)$ for all $\omega$.


\subsection{Multiplicativity and rigidity}
Historically, the first rigidity results are due to Atiyah and
Hirzebruch \cite{at-hi70}, and relate to the $T$-equivariant
$\chi_y$-genus and $\lowidehat{A}$-genus; their origins lie in the
Atiyah-Bott fixed point formula \cite{at-bo68}, which also acted
as a catalyst for the development of equivariant index theory.

Krichever \cite{kric74}, \cite{kric90} considers rational valued
genera $l$, and equivariant extensions
\smash{$l'\colon\varOmega^{U:T^k}_*\ra K^0(BT^k_+)\otimes\bQ$}
that arise from \smash{$l^{T^k}$} via the Chern-Dold character. If
$l$ may be realised on a stably complex manifold as the index of
an elliptic complex, then the value of $l'$ on any stably complex
$T^k$-manifold $M$ is a representation of $T^k$, which determines
$l'(M)$ in the representation ring $RU(T^k)\otimes\bQ$, and hence
in its completion $K^0(BT^k_+)\otimes\bQ$. Krichever's
interpretation of rigidity is to require that $l'(M)$ should lie
in the coefficient ring $K_0\otimes\bQ$ for every $M$. In the case
of an index, this amounts to insisting that the corresponding
$T^k$-representation is always trivial, and therefore conforms to
Atiyah and Hirzebruch's original notion \cite{at-hi70}.

The following definition extends that of \cite[Chapter 4]{h-b-j94}
for the oriented case. It applies to fibre bundles of the form
\smash{$M\to E\times_GM\stackrel{\pi}{\to} B$}, where $M$ and $B$
are closed, connected, and stably tangentially complex, $G$ is a
compact Lie group of positive rank whose action preserves the
stably complex structure on $M$, and $E\ra B$ is a principal
$G$-bundle. In these circumstances, $\pi$ is stably tangentially
complex, and $N\letbe E\times_GM$ inherits a canonical stably
complex structure.
\begin{definitions}\label{genusdefs}
  A genus $\ell\colon\varOmega^U_*\to R_*$ is \emph{multiplicative
  with respect to} the stably complex manifold
  $M$ whenever $\ell(N)=\ell(M)\ell(B)$ for any such bundle $\pi$;
  if this holds for every $M$, then $\ell$ is \emph{fibre
  multiplicative}.

The genus $\ell$ is \emph{$T^k$-rigid on} $M$ whenever
\smash{$\ell^{\,T^k}\colon\varOmega_*^{U:T^k}\lra
R_*[[u_1,\dots,u_k]]$} satisfies
\smash{$\ell^{\,T^k}(M,c_\tau)=\ell(M)$}; if this holds for every
$M$, then $\ell$ is \emph{$T^k$-rigid}.
\end{definitions}

We often indicate the rigidity of $\ell$ by referring to the
formal power series \smash{$\ell^{\,T^k}(M)$} as \emph{constant},
in which case $\ell(G_\omega(M))=0$ for every $|\omega|>0$, by
Corollary \ref{Gmrepsgm}.

In fact $\ell$ is $T^k$-rigid if and only if its rationalisation
$\ell_\bQ$ is rigid in Krichever's original sense. By way of
justification, we consider the universal example $ug$ of Example
\ref{cgenexs}(1); so \smash{$ug^{T^k}=\varPhi$} by construction,
and $ug_\bQ$ coincides with the Hurewicz genus
$hr_\bQ\colon\varOmega^U_*\ra
H_*(\MU;\bQ)\cong\varOmega^U_*\otimes\bQ$. The commutative square
\begin{equation}\label{ksq}
\begin{CD}
  \MU@>k^{M\?U}>>K\wedge\MU\\
  @Vh^{M\?U}\!VV@VV{h^{\?K}\!\wedge\£ 1}V\\
  H\wedge\MU@>{\;1\wedge k^{M\?U}\;}>>H\wedge K\wedge\MU
\end{CD}
\smallskip
\end{equation}
defines a homotopical genus $(k^{M\?U}\!,h^K\!\wedge 1)$, where
$k^{M\?U}$ and $h^K\colon K\ra H\wedge K$ induce the $K$-theoretic
Hurewicz homomorphism and the Chern character respectively. The genus
$k^{M\?U}\!$ is used to develop integrality results by tom Dieck
\cite[\S6]{tomd70}.

Applying \eqref{ksq} to $BT^k$ yields two distinct factorisations
of the homomorphism
\[
  \varOmega^U_*[[u_1,\dots,u_k]]\llra
  H_*(\MU;\bQ[z,z^{-1}])[[u_1,\dots,u_k]]\,,
\]
where $z$ denotes the image of the Bott periodicity element in
$H_2(K)\cong\bQ$. For any stably complex $T^k$-manifold $M$, the
lower-left factorisation maps $\varPhi(M)$ to
$\sum_{|\omega|\geq0}hr_\bQ(g_\omega(M))\,u^\omega$. On the other
hand, $k^{M\?U}_*\varPhi(M)$ lies in $K_*(\MU)[[u_1,\dots,u_k]]$,
where we may effect Krichever's change of orientation by following
\eqref{bmseries}, and rewriting $u_i$ as $\sum_jb^{K\wedge
  M\?U}_j\?x^K_i$ for $1\leq i\leq k$. Applying the Chern character
completes the right-upper factorisation, and identifies
$\varPhi(M)_\bQ$ with Krichever's formal power series
$ug'_\bQ(x_1^K,\dots,x_k^K)$ in
$\varOmega_*^U\otimes\bQ[z,z^{-1}][[x^K_1,\dots,x^K_k]]$.

But the change of orientation is invertible, so $\varPhi(M)_\bQ$
is a constant function of the $u_i$ if and only if
$ug'_\bQ(x_1^K,\dots,x_k^K)$ is a constant function of the
$x^K_i$, as claimed.

Krichever's choice of orientation is fundamental to his proofs of
rigidity, which bring techniques of complex analysis to bear on
Laurent series such as those of Proposition \ref{krich}. We follow
his example to a limited extent in Section \ref{apfuex}, where we
interpret certain \smash{$\ell^{\,T^k}(M)$} as formal powers
series over $\bC$, and study their properties using classical
functions of a complex variable. This viewpoint offers a powerful
computational tool, even in the homotopical context.

For rational genera in the oriented category, Ochanine
\cite[Proposition 1]{ocha88} proves that fibre multiplicativity
and rigidity are equivalent. In the toric case, we have the
following stably complex analogue, whose conclusions are integral.
It refers to bundles \smash{$E\times_GM\stackrel{\pi}{\ra}B$} of
the form required by Definition \ref{genusdefs}, where $G$ has
maximal torus $T^k$ with $k\geq 1$.
\begin{theorem}\label{multandrig}
  If the genus $\ell$ is $T^k$-rigid on $M$, then it is mutiplicative
  with respect to $M$ for bundles whose structure group $G$ has the
  property that $\varOmega^*_U(BG)$ is torsion-free; on the other
  hand, if $\ell$ is multiplicative with respect to a stably
  tangentially complex $T^k$-manifold $M$, then it is $T^k$-rigid on
  $M$.
\end{theorem}
\begin{proof}
Let $\ell$ be $T^k$-rigid, and consider the pullback squares
\begin{equation}\label{pulsqs}
\begin{CD}
  E\times_GM@>f'>>EG\times_GM@<i'<<ET^k\times_{T^k}M\\
  @V\pi VV@V{\pi^G}VV@V{\pi^{T^k}}VV\\
  B^{2b}@>f>>BG@<i<<BT^k,
\end{CD}
\smallskip
\end{equation}
where $\pi^G$ is universal, $i$ is induced by inclusion, $f$
classifies $\pi$, and $E\times_GM=N$. If $1$ is the unit in
$\varOmega^0((EG\times_GM)_+)$, then $\pi^G_*1=[M]\cdot 1+\beta$,
where $\beta$ lies in the reduced group $\varOmega^{-2n}(BG)$ and
satisfies $\pi_*1=[M]\cdot1+f^*\beta$ in $\varOmega^{-2n}(B_+)$.
Applying the Gysin homomorphism associated to the augmentation map
$\epsilon^B\colon B\ra*$ yields
\begin{equation}\label{mult}
[N]\;=\;\epsilon^B_*\pi_*1\;=\; [M][B]+\epsilon^B_*f^*\beta
\end{equation}
in $\varOmega_{2(n+b)}$; so $\ell(N)=\ell(M)\ell(B)+
\ell(\epsilon^B_*f^*\beta)$. Moreover,
\smash{$i^*\beta=\sum_{|\omega|>0}g_\omega(M)u^\omega$} in
$\varOmega^{-2n}(BT)$, so $\ell(i^*\beta)=0$. The assumptions on $G$
ensure that $i^*$ is monic, and hence that $\ell(\beta)=0$ in
$\varOmega^*_U(BG)\otimes_\ell R_*$. Multiplicativity then follows
from \eqref{mult}.

Conversely, suppose that $\ell$ is multiplicative with respect to $M$,
and consider the manifold $G_\omega(M)$ of Corollary
\ref{Gmrepsgm}. By Definition \ref{Gomegadef}, it is the total space
of the bundle
$\big((S^3)^\omega\times_{T^\omega}\!T^k\big)\times_{T^k}\!M\ra
B_\omega$, which has structure group $T^k$; therefore $\ell(G_\omega(M))=0$,
because $B_\omega$ bounds for every $|\omega|>0$. So $\ell$ is
$T^k$-rigid on $M$.
\end{proof}

\begin{remark}
  We may define $\ell$ to be \emph{$G$-rigid} when $\ell(\beta)=0$, as
  in the proof of
  Theorem \ref{multandrig}.
  By choosing a suitable
  subcircle $T<T^k$, it follows that $T$-rigidity implies $G$-rigidity
  for any $G$ such that $\varOmega^*_U(BG)$ is
  torsion-free.
\end{remark}

\begin{example}\label{rigidex}
The signature $sg$ of Examples \ref{cgenexs}(6) is fibre
multiplicative over any simply connected base \cite{h-b-j94}, and so
is rigid.
\end{example}

%
%
%
%
%
%
%
%
%

\section{Isolated fixed points}\label{isfipo}

In this section and the next, we focus on stably tangentially complex
$T^k$-manifolds $(M^{2n},a,c_\tau)$ for which the fixed points $x$ are
isolated; in other words, the fixed point set $\fix(M)$ is finite. We
proceed by adapting Quillen's methods to describe $\varPhi(\pi)$ for
any stably tangentially complex $T^k$-bundle $\pi\colon E\ra X$, and
deducing a localisation formula for $\varPhi(M)$ in terms of fixed
point data. This extends Krichever's formula \eqref{kf}. We give
several illustrative examples, and describe the consequences for
certain non-equivariant genera and their $T^k$-equivariant
extensions. A condensed version of parts of this section appears in
\cite{bu-ra07} for the case $k=1$.

For any fixed point $x$, the representation $r_x\colon T^k\to
\GL(l,\bC)$ associated to \eqref{rtpi} decomposes the fibre $\xi_x$ as
$\mathbb{C}^n\oplus\mathbb{C}^{l-n}$, where $r_x$ has no trivial
summands on $\bC^n$, and is trivial on $\bC^{l-n}$. Also, $c_{\tau,x}$
induces an orientation of the tangent space $\tau_x(M)$.
\begin{definition}\label{defsign}
For any $x\in\fix(M)$, the \emph{sign} $\varsigma(x)$ is $+1$ if the
isomorphism
\[
  \tau_x(M)\stackrel{i}{\llra}\tau_x(M)\oplus
  \mathbb{R}^{2(l-n)}\stackrel{c_{\tau\!,x}}{\llra}\xi_x
  \cong\mathbb{C}^n\oplus\mathbb{C}^{l-n}\stackrel{p}{\llra}
  \mathbb{C}^n
\]
respects the canonical orientations, and $-1$ if it does not; here
$i$ and $p$ are the inclusion of, and projection onto the first
summand, respectively.
\end{definition}

So $\varsigma(x)$ compares the orientations induced by $r_x$ and
$c_{\tau,x}$ on $\tau_x(M)$, and if $M$ is almost complex then
$\varsigma(x)=1$ for every $x\in\fix(M)$. The non-trivial summand
of $r_x$ decomposes into $1$-dimensional representations as
\begin{equation}\label{repre}
  r_{x,1}\oplus\ldots\oplus r_{x,n},
\end{equation}
and we write the integral \emph{weight vector} of $r_{x,j}$ as
$w_j(x)\letbe(w_{j,1}(x),\dots,w_{j,k}(x))$, for $1\leq j\leq n$. We
refer to the collection of signs $\varsigma(x)$ and weight vectors
$w_j(x)$ as the \emph{fixed point data} for $(M,c_\tau)$.

Localisation theorems for stably complex $T^k$-manifolds in
equivariant generalised cohomology theories appear in tom Dieck
\cite{tomd70}, Quillen \cite[Proposition 3.8]{quil71},
Krichever~\cite[Theorem 1.1]{kric74}, Kawakubo~\cite{kawa80}, and
elsewhere. We prove our Corollary~\ref{evutg} by interpreting their
results in the case of isolated fixed points, and identifying the
signs explicitly. To prepare, we recall the most important details of
\cite[Proposition 3.8]{quil71}, in the context of \eqref{phitan} and
\eqref{phinorm}.

Given any representative $\pi\colon E\ra X$ for $\pi$ in
\smash{$\varOmega^{-2n}_{U:T^k}(X_+)$}, we approximate
$\varPhi(\pi)$ by the projection map $\pi(q)\letbe 1\times_{T^k}\pi$
of the smooth fibre bundle
\[
  F\llra \varPi\?S(q)\times_{T^k}E\llra \varPi\?S(q)\times_{T^k}X
\]
for suitably large $q$. So $\pi(q)$ is complex oriented by the
corresponding approximation to \eqref{phinorm}, whose associated
embedding $h(q)\colon\varPi\?S(q)\times_{T^k}E\ra
\varPi\?S(q)\times_{T^k}(X\times V)$ is the Borelification of
$(\pi,i)$. Then there exists a commutative square
\begin{equation}\label{quilsq}
\begin{CD}
  \varPi\?P(q)\times\fix(E)\!\!@>1\times_{T^k}r_E>>
  \varPi\?S(q)\times_{T^k}E\\
  @V{\pi'(q)}VV@VV{\pi(q)}V\\
  \varPi\?P(q)\times\fix(X)\!\!@>1\times_{T^k}r_X>>
  \varPi\?S(q)\times_{T^k}X\,,
\end{CD}
\end{equation}
where $r_E$ and $r_X$ are the inclusions of the fixed point
submanifolds, and $\pi'(q)$ is the restriction of $1\times\pi$.

Let $\mu'(q)$ and $\mu(V;q)$ denote the respective decompositions of
the complex $T^k$-bundles $(1\times_{T^k}r_E)^*\nu(h(q))$ and
$\varPi\?S(q)\times_{T^k}V\ra\varPi\?P(q)$ into eigenbundles, given
by non-trivial $1$-dimensional representations of~$T^k$. They have
respective Euler classes $e(\mu'(q))$ in
$\varOmega^{2(|V|-n)}_U(\varPi\?P(q)\times\fix(E))$ and
$e(\mu(V;q))$ in $\varOmega^{2|V|}_U(\varPi\?P(q))$. Following
\cite[Remark 3.9]{quil71}, we then define
\begin{equation}\label{defemupiq}
e(\mu_{\pi(q)})\;\letbe\;e(\mu'(q))/e(\mu(V,q)).
\end{equation}
This expression may be interpreted as a $-2n$ dimensional element
of the localised ring
$\varOmega_U^*(\varPi\?P(q)\times\fix(E))[e(\mu(V,q))^{-1}]$, or
as a formal quotient whose denominator cancels wherever it appears
in our formulae. In either event, $e(\mu_{\pi(q)})$ is determined
by the stably complex structure $c_\tau(\pi)$ on $\tau_F(E)$, but
is independent of the choice of complementary complex orientation.

Quillen shows that, for any $z\in \varOmega^*_U(\varPi\?
S(q)\times_{T^k}E)$, the equation
\begin{equation}\label{quilf}
  (1\times_{T^k}r_X)^*\pi(q)_*z\;=\;
\pi'(q)_*\big(e(\mu_{\pi(q)})\cdot(1\times_{T^k}r_E)^*z\big)
\end{equation}
holds in $\varOmega^*_U(\varPi\?P(q)\times\fix(X))$. In evaluating
such terms we may take advantage of the K\"unneth formula, because
$\varOmega^*_U(\varPi\?P(q))$ is free over $\varOmega_*^U$.

When the fixed points of $\pi\colon E\ra X$ are isolated, this
analysis confirms that the notion of fixed point data extends to
$x\in\fix(E)$; the signs $\varsigma(x)$ and weight vectors
$w_j(x)$ are determined by the stably complex $T^k$-structure on
$\tau_F(E)$, for $1\leq j\leq n$. Every such $x$ lies in a fibre
$F_y\letbe\pi^{-1}(y)$, for some fixed point $y\in\fix(X)$.

\begin{theorem}\label{evutgx}
  For any stably tangentially complex $T^k$-bundle $F\ra
  E\stackrel{\pi}{\ra}X$ with isolated fixed points, the equation
\begin{equation}\label{lfxX}
  {\varPhi(\pi)\big|}_{\fix(X)}
\;=\;\sum_{\fix(X)}\;\sum_{\fix(F_y)}
\varsigma(x)\prod_{j=1}^n\frac1{[w_j(x)](u)}
\end{equation}
is satisfied in $\bigoplus_{\fix(X)}\varOmega^{-2n}_U(BT^k_+\times
y)$, where $y$ and $x$ range over $\fix(X)$ and $\fix(F_y)$
respectively .
\end{theorem}
\begin{proof}
  Set $z=1$ in \eqref{quilf}. By definition, $\pi(q)_*(1)$ is
  represented by the complex oriented map $\pi(q)$ in
  $\varOmega^{-2n}_U(\varPi\?S(q)\times_{T^k}X)$, and
  approximates $\varPhi(\pi)$; denote it by
  $\varPhi(\pi;q)$. Its restriction to $\varPi\?P(q)\times y$ for
  any fixed point $y$ is represented by the induced projection
  $\pi_y(q)\colon\varPi\?S(q)\times_{T^k}F_y\ra\varPi\?P(q)\times
  y$, which is complex oriented via the restriction
\[
1\times_{T^k}i_y\colon
\varPi\?S(q)\times_{T^k}F_y\lra\varPi\?S(q)\times_{T^k}V_y
\]
of the Borelification of $(\pi,i)$ to $F_y$ .

The commutative square \eqref{quilsq} breaks up into components of
the form
\[
\begin{CD}
  \varPi\?P(q)\times\fix(F_y)\!\!@>1\times_{T^k}r_{F_y}>>
  \varPi\?S(q)\times_{T^k}F_y\\
  @V{\pi'_y(q)}VV@VV\pi_y(q)V\\
  \varPi\?P(q)\times y\!\!@>1>>
  \varPi\?P(q)\times y
\end{CD}\quad,
\]
on which \eqref{quilf} reduces to
\begin{equation}\label{newrestrict}
  {\varPhi(\pi;q)|}_y\;=\;\pi'_y(q)_*\big(e(\mu_{\pi_y(q)})\big)
\end{equation}
in $\varOmega^{-2n}_U(\varPi\?P(q)\times y)$. But $\fix(F_y)$ is
finite, so the right hand side may be replaced by $\sum_x
e(\mu(x,y))$, as $x$ ranges over $\fix(F_y)$; here $e(\mu(x,y))$
denotes the quotient $e(\mu'_y(q)|_x)/e(\mu(V_y,q)|_x)$ of
\eqref{defemupiq}, where $\mu'_y(q)|_x$ and $\mu(V_y,q)|_x)$ are the
respective decompositions of the $T^k$-bundles
$1\times_{T^k}r_{F_y}^*\nu(i_y)$ and $\varPi\?S(q)\times_{T^k}V_y\ra
\varPi\?P(q)$ into sums of complex line bundles, upon restriction to
the fixed point $x$.

The restriction of the equivariant isomorphism
$\tau(V)|_{F_y}\cong\tau(F_y)\oplus\nu(i_y)$ to $x$ respects these
decompositions, and reduces to
$\mu(V_y,q)|_x\cong\tau_x(F_y)\oplus\mu'_y(q)|_x$. Taking Euler
classes and applying Definition \ref{defsign} and \eqref{repre} then
gives
\[
  e\big(\mu(V_y,q)|_x\big)\;=\;\varsigma(x)\cdot
  e(r_{x,1}\oplus\ldots\oplus r_{x,n})\cdot
  e(\mu'_y(q)|_x)
\]
in $\varOmega_U^{2|V|}(\varPi\?P(q)\times y)$. So
$e(\mu(x,y))\cdot e(r_{x,1}\oplus\ldots\oplus
r_{x,n})=\varsigma(x)$, and \eqref{newrestrict} becomes
\[
  {\varPhi(\pi;q)|}_y\;=\;\sum_x\varsigma(x)
  \frac1{e(r_{x,1}\oplus\ldots\oplus r_{x,n})}.
\]

To complete the proof, recall from \eqref{powersys} that
\[
  e(r_{x,j})\;=\;[w_j(x)](u)
\]
in $\varOmega_U^{-2n}(\varPi\?(q))$, for $1\le j\le n$; then let
$q\ra\infty$, and sum over $y\in\fix(X)$.
\end{proof}
\begin{corollary}\label{evutg}
  For any stably tangentially complex $M^{2n}$ with
  isolated fixed points, the equation
\begin{equation}\label{lfx}
  \varPhi(M)\;=\;\sum_{\fix(M)}\!
  \varsigma(x) \prod_{j=1}^n\frac1{[w_j(x)](u)}
\end{equation}
is satisfied in $\varOmega^{-2n}_U(BT^k_+)$.
\end{corollary}
\begin{proof}
Apply Theorem \ref{evutgx} to the case $E=M$ and $X=*$.
\end{proof}

\begin{remark}\label{genkrich}
If the structure $c_\tau$ is almost complex, then $\varsigma(x)=1$
for all fixed points $x$, and \eqref{lfx} reduces to Krichever's
formula \cite[(2.7)]{kric90}).
\end{remark}

The left-hand side of \eqref{lfx} lies in
$\varOmega^U_*[[u_1,\dots,u_k]]$, whereas the right-hand side appears
to belong to an appropriate localisation. It follows that all terms of
negative degree must cancel, thereby imposing substantial restrictions
on the fixed point data. These may be made explicit by rewriting
\eqref{wjxmoddec} as
\begin{equation}\label{psit}
[w_j(x)](tu)\;\equiv\;(w_{j,1}u_1+\dots +w_{j,k}u_k)t \mod (t^2)
\end{equation}
in $\varOmega^U_*[[u_1,\dots,u_k,t]]$, and then defining the formal
power series \smash{$\sum_l\cf_lt^l$} to be
\begin{equation}\label{tseries}
t^n\varPhi(M)(tu)\;=\;
\sum_{\fix(M)}\!\varsigma(x)\prod_{j=1}^n\frac{t}{[w_j(x)](tu)}
\end{equation}
over the localisation of $\varOmega^U_*[[u_1,\dots,u_k]]$.
\begin{proposition}\label{cfrels}
The coefficients $\cf_l$ are zero for $\,0\leq l<n$, and satisfy
\[
\cf_{n+m}\;=\;\sum_{|\omega|=m}g_\omega(M)u^\omega
\]
for $m\geq 0$; in particular, $\cf_n=[M]$.
\end{proposition}
\begin{proof}
Combine the definitions of $\cf_l$ in \eqref{tseries} and
$g_\omega$ in \eqref{defgomega}.
\end{proof}
\begin{remarks}\hfill\\
\indent (1) The equations $\cf_l=0$ for $0\leq l<n$ are the
$T^k$-analogues of the \emph{Conner--Floyd relations} for
$\mathbb{Z}/p$-actions~\cite[Appendix~4]{novi67}; the extra equation
$\cf_n=[M]$ provides an expression for the cobordism class of $M$ in
terms of fixed point data. This is important because every element of
$\varOmega^U_*$ may be represented by a stably tangentially complex
$T^k$-manifold with isolated fixed points \cite{b-p-r07}.

(2) There are parametrised versions of the Conner-Floyd relations
associated to any $\varPhi(\pi)$, which arise from the individual
fibres $F_y$. As we show in Example \ref{tarassq}, they provide
necessary conditions for the existence of stably tangentially complex
$T^k$-equivariant bundles.
\end{remarks}

Theorem \ref{evutgx} applies to fibre bundles of the form
\begin{equation}\label{thg}
H/T^k\lra G/T^k\stackrel{\pi}{\lra}G/H
\end{equation}
for any compact connected Lie subgroup $H<G$ of maximal rank $k$, by
extrapolating the methods of \cite{bu-te08}. Corollary \ref{evutg}
applies to homogeneous spaces $G/H$ with nonzero Euler
characteristic, to toric and quasitoric manifolds~\cite{bu-pa02},
and to many other families of examples.

\begin{examples}\label{isfipoexas}\hfill\\
  \indent (1) For $U(1)\times U(2)<U(3)$, \eqref{thg} is the
  $T^3$-bundle $\bC P^1\ra U(3)/T^3\stackrel{\pi}{\ra}\bC P^2$; it is
  the projectivisation of $\eta^\perp$, and all three manifolds are
  complex. Moreover, $U(3)/T^3$ has 6 fixed points and $\bC P^2$ has
  3, so $\pi$ represents an element of $\varOmega^{-2}_{U:T^3}(\bC
  P^2)$, and $\varPhi(\pi)$ lies in
  $\varOmega^{-2}_U((ET^3\times_{T^3}\bC P^2)_+)$.

The 6 sets of weight vectors are singletons, given by the action
of the Weyl group $\varSigma_3$ on the roots of $U(3)$, and split
into 3 pairs, indexed by the cosets of $\varSigma_2<\varSigma_3$.
These pairs are $\pm(-1,1,0)$,\; $\pm(-1,0,1)$,\, and
$\pm(0,-1,1)$\, respectively, corresponding to tangents along the
fibres. So \eqref{lfxX} evaluates $\varPhi(\pi)|_{\fix(\bC P^2)}$
as
\[
\Big(\frac1{F(\bar u_1,u_2)}+\frac1{F(u_1,\bar u_2)}\Big)\;+\;
\Big(\frac1{F(\bar u'_1,u'_3)}+\frac1{F(u'_1,\bar u'_3)}\Big)\;+\;
\Big(\frac1{F(\bar u''_2,u''_3)}+\frac1{F(u''_2,\bar u''_3)}\Big)
\]
in $\varOmega^{-2}_U(BT^3_+)\oplus\varOmega^{-2}_U(BT^3_+)'
\oplus\varOmega^{-2}_U(BT^3_+)''$, where $F(u_1,u_2)$ denotes the
universal formal group law $[(1,1)](u_1,u_2)$, and $\bar u_j$
denotes $[-1]u_j$ for any $j$.

(2) The exceptional Lie group $G_2$ contains $\SU(3)$ as a subgroup
of maximal rank, and induces a canonical $T^2$-invariant almost
complex structure on the quotient space $S^6$. This has two fixed
points, with weight vectors $(1,0)$, $(0,1)$, $(-1,-1)$ and
$(-1,0)$, $(0,-1)$, $(1,1)$ respectively. So Corollary \ref{evutg}
evaluates $\varPhi(S^6,c_\tau)$ as
\[
\frac{1}{u_1u_2F(\bar u_1,\bar u_2)}\:+\:\frac{1}{\bar u_1\bar
  u_2F(u_1,u_2)}
\]
in $\varOmega^{-6}_U(BT^2)$. The coefficient $c\?f_0$ of Proposition
\ref{cfrels} is $(1-1)/u_1u_2(u_1+u_2)$, and is visibly zero; the
cases $c\?f_1$ and $c\?f_2$ are more intricate.
\end{examples}

Given any homotopical genus $(t^D,v)$, we may adapt both Theorem
\ref{evutgx} and Corollary \ref{evutg} to express the genus
$t^D_*\varPhi$ in terms of fixed point data and the $D$-theory Euler
class $e^D\letbe t^D_*e$. Evaluating $t^E_*\varPhi$ leads to major
simplifications, because $F_{(t^D,v)}$ may be linearised over $E_*$; the
tangential form is also useful.
\begin{proposition}\label{evpargenus}
  For any stably tangentially complex $T^k$-bundle $F\ra
  E\stackrel{\pi}{\ra}X$ with isolated fixed points,
  $t^E_*\varPhi(\pi)\big|_{\fix(X)}$ takes the value
\begin{equation}\label{hlfxX}
\sum_{\fix(E)}\varsigma(x)\prod_{j=1}^n\frac1{b^E(w_j(x)\cdot
u)} \;=\; \sum_{\fix(E)}\varsigma(x)
\prod_{j=1}^n\frac{(a^E)_+(w_j(x)\cdot u)}{w_j(x)\cdot u}
\end{equation}
in $\bigoplus_{\fix(X)}E^{-2n}(BT^k_+\times y)$.
\end{proposition}
\begin{proof}
  Apply the Thom class $t^E$ to Theorem \ref{evutgx}, using the
  relation $e^E=b^E(e^H)$ in $E^2(\bC P^\infty)$. Then the proof is
  completed by the linearising equations
\[
  e^E(r_{x,j})=b^E(e^H(r_{x,j}))=b^E(w_j(x)\cdot u),
\]
together with the identity $1/b_+^E(x^H)=
a^E_+(x^H)$ in $E^0(\bC P^\infty_+)$.
\end{proof}
\begin{corollary}\label{evgenus}
For any stably tangentially complex $M^{2n}$ with isolated fixed
points, the equation
\[
  t^E_*\varPhi(M)\;=\;\sum_{\fix(M)}\!
  \varsigma(x)\prod_{j=1}^n\frac1{b^E(w_j(x)\cdot u)}
\]
is satisfied in $E^{-2n}(BT^k_+)$.
\end{corollary}
\begin{remarks}\label{3rems}\hfill\\
\indent (1) The corresponding analogue of \eqref{tseries} is the
series $\sum_lcf^E_lt^l$, defined by
\begin{equation}\label{htseries}
  \sum_{\fix(E)}\varsigma(x)\prod_{j=1}^n\frac{t}{b^E(tw_j(x)\cdot
  u)}\;=\; \sum_{\fix(E)}\varsigma(x) \prod_{j=1}^n
  \frac{a^E_+(tw_j(x)\cdot u)}{w_j(x)\cdot u}
\end{equation}
over the localisation of $E_*[[u_1,\dots,u_k]]$. So
$cf_n^E$ is the first nonzero coefficient, and computes
the nonequivariant genus $t^E_*(M)$. Similarly, the constant term of
\eqref{hlfxX} is $\sum_{\fix(X)}t^E_*(F_y)$, and the principal part
is $0$.

(2) The Hurewicz genus $hr(\varPhi(\pi))$ of Examples~\ref{cgenexs}(2)
is used implicitly in many of the calculations of \cite{bu-te08}, and
lies in $\bigoplus_{\fix(X)}(H\wedge\MU)^{-2n}(BT^k_+\times y)$.
\end{remarks}

\begin{example}\label{agcfrels}
  The augmentation genus $ag(\varPhi(\pi))$ of
  Example~\ref{cgenexs}~(1) may be computed directly from Theorem
  \ref{evutgx}. It lies in $\bigoplus_{\fix(X)}H^{-2n}(BT^k_+\times
  y)$, and is therefore zero for every $\pi$ of positive fibre
  dimension. It follows that
\begin{equation}\label{agzero}
\sum_{\fix(E)\!\!}\varsigma(x)\prod_{j=1}^n\frac1{w_j(x)\cdot u}\;=\;0
\end{equation}
for all $M^{2n}$, because $e^H(r_{x,j})=w_j(x)\cdot u$ in
$H^2(BT^k)$. In case $\pi\colon\bC P^n\ra *$ and $T^{n+1}$ acts on
$\bC P^n$ homogeneous coordinatewise, every fixed point $x_0$, \dots,
$x_n$ has a single nonzero coordinate. So the weight vector $w_j(x_k)$
is $e_j-e_k$ for $0\le j\neq k\le n$, and every $\varsigma(x_k)$ is
positive; thus \eqref{agzero} reduces to
\begin{equation}\label{agcpnstand}
  \sum_{k=0}^n\;\;\prod_{0\le j\neq k\le
  n}\!\frac1{u_j-u_k}\;=\;0\,.
\end{equation}
This is a classical identity, and is closely related to formulae in
\cite{h-b-j94}[\S 4.5].
\end{example}

\begin{examples}\hfill\\
  \indent (1) The formula associated to Remarks
  \ref{3rems}(2) lies behind \cite[Theorem 9]{bu-te08}. The canonical
  $T^n$-action on the complex flag manifold $U(n)/T^n$ has $n!$ fixed
  points; these include the coset of the identity, whose weight vectors
  are the roots of $U(n)$, given by $e_i-e_j$ for $1\le i<j\le n$. The
  action of the Weyl group $\varSigma_n$ permutes the coordinates, and
  yields the weight vectors of the other fixed points. All have
  positive sign. Substituting this fixed point data into \eqref{htseries}
  gives
\[
  hr(U(n)/T^n)\;=\;\sum_{\varSigma_n}
  \mathop{\mathrm{sign}}(\rho)P_{\rho\delta}(a_1,\ldots,a_m),
\]
in $(H\wedge\MU)_{n(n-1)}$, where $\delta=(n-1,n-2,\ldots,1,0)$,
$m=n(n-1)/2$, and $\rho$ ranges over $\varSigma_n$; the
$P_\omega(a_1,a_2,\dots)$ are polynomials in
$(H\wedge\MU)_{2|\omega|}$, defined by
\[
  \prod_{1\le i<j\le n}\!a_+(t(u_i-u_j))\;=\;
  1+\sum_{|\omega|>0}P_\omega(a_1,a_2,\dots)t^{|\omega|}u^\omega.
\]

(2) Similar methods apply to $(S^6,c_\tau)$ of Example
\ref{isfipoexas}(2), and show that
\[
hr(S^6,c_\tau)=2(a_1^3-3a_1a_2+3a_3)
\]
in $H_6(\MU)$, as in \cite{bu-te08}. This result may also be read off
from calculations in \cite{ray74}.
\end{examples}

%
%
%
%
%
%
%
%

\section{Quasitoric manifolds}\label{quma}

As shown in~\cite{bu-ra01}, Davis and Januszkiewicz's
\emph{quasitoric manifolds} \cite{da-ja91} provide a rich supply of
stably tangentially complex $T^n$-manifolds $(M^{2n},a,c_\tau)$ with
isolated fixed points. In this section we review the basic theory,
and highlight the presentation of an \emph{omnioriented} quasitoric
manifold in purely combinatorial terms. Our main goal is to compute
the fixed point data directly from this presentation, and therefore
to evaluate the genus $\varPhi(M)$ combinatorially by applying the
formulae of Section \ref{isfipo}. We also describe the
simplifications that arise in case $M$ is a non-singular projective
\emph{toric variety}. Our main source of background information is
\cite{b-p-r07}, to which we refer readers for additional details.

An $n$-dimensional convex polytope $P$ is the bounded intersection of
$m$ irredundant halfspaces in $\bR^n$. The bounding hyperplanes $H_1$,
\dots, $H_m$ meet $P$ in its \emph{facets} $F_1$, \dots, $F_m$, and
$P$ is \emph{simple} when every face of codimension $k$ may be
expressed as $F_{j_1}\cap\dots\cap F_{j_k}$, for $1\leq k\leq n$. In
this case, the facets of $P$ are \emph{finely ordered} by insisting that
$F_1\cap\dots\cap F_n$ defines the \emph{initial vertex} $v_1$.

Two polytopes are \emph{combinatorially equivalent} whenever their
face posets are isomorphic; the corresponding equivalence classes
are known as \emph{combinatorial polytopes}. At first sight, several
of our constructions depend upon an explicit realisation of $P$ as a
convex subset of $\bR^n$. In fact, they deliver equivalent output
for combinatorially equivalent input, so it suffices to interpret
$P$ as combinatorial whilst continuing to work with a geometric
representative when convenient. For further details on this point,
we refer to \cite{bu-pa02} and \cite[Corollary~4.7]{bo-me06}, for
example. The facets of a combinatorial simple polytope $P$ may also
be finely ordered, and an \emph{orientation} is then an equivalence
class of permutations of $F_1$, \dots, $F_n$.

There is a canonical affine transformation $i_P\colon P\ra\brpm$
into the positive orthant of $\bR^m$, which maps a point of $P$ to
its $m$-vector of distances from the hyperplanes $H_1$, \dots,
$H_m$. It is an embedding of manifolds with corners, and is
specified by $i_P(x)=A_Px+b_P$, where $A_P$ is the $m\times n$
matrix of inward pointing unit normals to $F_1$, \dots, $F_m$. The
\emph{moment-angle manifold} $\zp$ is defined as the pullback
\begin{equation}\label{cdiz}
\begin{CD}
  \zp @>i_Z>>\bC^m\\
  @V\mu_P VV\hspace{-0.2em} @VV\mu V @.\\
  P @>i_P>>\brpm
\end{CD}\quad,
\end{equation}
where $\mu(z_1,\ldots,z_m)$ is given by $(|z_1|^2,\ldots,|z_m|^2)$.
The vertical maps are projections onto the $T^m$-orbit spaces, and
$i_Z$ is a $T^m$-equivariant embedding. It follows from Diagram
\eqref{cdiz} that $\zp$ may be expressed as a complete intersection
of $m-n$ real quadratic hypersurfaces in $\bC^m$~\cite{b-p-r07}, so it
is smooth and equivariantly framed. We therefore acquire a
$T^m$-equivariant isomorphism
\begin{equation}\label{tauzp}
  \tau(\zp)\oplus\nu(i_Z)\stackrel{\cong}{\llra}\zp\times\bC^m,
\end{equation}
where $\tau(\zp)$ is the tangent bundle of $\zp$ and $\nu(i_Z)$ is
the normal bundle of $i_Z$. Construction \eqref{cdiz} was
anticipated by work of L\'opez de Medrano \cite{lodeme89}.

It is important to describe the action of $T^m$ on $\zp$ in more
detail. For any vector $z$ in $\bC^m$, let $I(z)\subseteq [m]$
denote the set of subscripts $j$ for which the coordinate $z_j$ is
zero; and for any point $p$ in $P$, let $I(p)\subseteq [m]$ denote
the set of subscripts for which $p$ lies in $F_j$. The isotropy
subgroup of $z$ under the standard action on $\bC^m$ is the
\emph{coordinate subgroup} $T_{I(z)}\leq T^m$, consisting of those
elements $(t_1,\dots,t_m)$ for which $t_j=1$ unless $j$ lies in
$I(z)$. We deduce that the isotropy subgroup of $z$ in $\zp$ is
$T_z=T_{I(\mu_P(z))}\leq T^m$, using the definitions of $i_P$ and
$\mu$.
\begin{definition}\label{data}
A {\it combinatorial quasitoric pair} $(P,\varLambda)$ consists of
an oriented combinatorial simple polytope $P$ with finely ordered
facets, and an integral $n\times m$ matrix $\varLambda$, whose
columns $\lambda_i$ satisfy
\begin{enumerate}
\item
$(\lambda_1\;\dots\;\lambda_n)$ is the identity submatrix $I_n$
\item
$\det(\lambda_{j_1}\;\dots\;\lambda_{j_n})=\pm 1$ whenever
$F_{j_1}\cap\dots\cap F_{j_n}$ is a vertex of $P$.
\end{enumerate}
A matrix obeying (1) and (2) is called \emph{refined}, and (2) is
\emph{Condition}~$(*)$ of~\cite{da-ja91}.
\end{definition}

Given a combinatorial quasitoric pair, we use the matrix
$\varLambda$ to construct a quasitoric manifold
$M=M^{2n}(P,\varLambda)$ from $\zp$. The first step is to
interpret the matrix as an epimorphism $\varLambda\colon T^m\ra
T^n$, whose kernel we write as $K=K(\varLambda)$. The latter is
isomorphic to $T^{m-n}$ by virtue of Condition $(*)$. Furthermore,
$\varLambda$ restricts monomorphically to any of the isotropy
subgroups $T_z<T^m$, for $z\in\zp$; therefore $K\cap T_z$ is
trivial, and $K<T^m$ acts freely on $\zp$. This action is also
smooth, so the orbit space $M=M(P,\varLambda)\letbe\zp/K$ is a
$2n$--dimensional smooth manifold, equipped with a smooth action
$a$ of the quotient $n$-torus $T^n\cong T^m/K$. By construction,
$\mu_P$ induces a projection $\pi\colon M\to P$, whose fibres are
the orbits of~$a$. Davis and Januszkiewicz \cite{da-ja91} also
show that $a$ is \emph{locally standard}, in the sense that it is
locally isomorphic to the coordinatewise action of $T^n$ on
$\bC^n$.

In fact the pair $(P,\varLambda)$ invests $M$ with additional
structure, obtained by factoring out the decomposition \eqref{tauzp}
by the action of $K$. We obtain an isomorphism
\begin{equation}\label{taum}
  \tau(M)\oplus\left(\xi/K\right)\oplus
  \left(\nu(i_Z)/K\right)\stackrel{\cong}{\llra}
  \zp\times_{K}\bC^m
\end{equation}
of real $2m$--plane bundles, where $\xi$ denotes the $(m-n)$--plane
bundle of tangents along the fibres of the principal $K$-bundle
$\zp\to M$. The right hand side of \eqref{taum} is isomorphic to a
sum $\bigoplus_{i=1}^m\rho_i$ of complex line bundles, where
$\rho_i$ is defined by the projection
\[
  \zp\times_K\bC_i\lra M
\]
associated to the action of $K$ on the $i$th coordinate subspace.
Both $\xi$ and $\nu(i_Z)$ are equivariantly trivial over $\zp$, so
$\xi/K$ and $\nu(i_Z)/K$ are also trivial; explicit framings are
given in \cite[Proposition 4.5]{b-p-r07}. Hence \eqref{taum}
reduces to an isomorphism
\begin{equation}\label{tansumrho}
  \tau(M)\oplus\bR^{2(m-n)}
  \stackrel\cong\llra\rho_1\oplus\ldots\oplus\rho_m\,,
\end{equation}
which specifies a canonical $T^n$-equivariant stably tangentially
complex structure $c_\tau=c_\tau(P,\varLambda)$ on $M$. The
underlying smooth structure is that described above.

In order to make an unambiguous choice of isomorphism
\eqref{tansumrho}, we must also fix an orientation of $M$. This is
determined by substituting the given orientation of $P$ into the
decomposition
\begin{equation}\label{orientM}
  \tau_{(p,1)}(M)\;\cong\;\tau_p(P)\oplus\tau_1(T^n)
\end{equation}
of the tangent space at $(p,1)\in M$, where $p$ lies in the interior
of $P$ and $1\in T^n$ is the identity element. Together with the
orientations underlying the complex structures on the bundles
$\rho_i$, we obtain a sequence of $m+1$ orientations; following
\cite{b-p-r07}, we call this the \emph{omniorientation} of $M$
\emph{induced by} $(P,\varLambda)$.

As real bundles, the $\rho_i$ have an alternative description in
terms of the \emph{facial} (or \emph{characteristic})
\emph{submanifolds} $M_i\subset M$, defined by $\pi^{-1}(F_i)$ for
$1\leq i\leq m$. With respect to the action $a$, the isotropy
subgroup of $M_i$ is the subcircle of $T^n$ defined by the column
$\lambda_i$. Every $M_i$ embeds with codimension $2$, and has an
orientable normal bundle $\nu_i$ that is isomorphic to the
restriction of $\rho_i$; furthermore, $\rho_i$ is trivial over
$M\setminus M_i$. So an omniorientation may be interpreted as a
sequence of orientations for the manifolds $M$, $M_1$, \dots, $M_m$.
\begin{definition}\label{oqmofpair}
The stably tangentially complex $T^n$-manifold $M(P,\varLambda)$ is
the omnioriented quasitoric manifold \emph{corresponding to the
pair} $(P,\varLambda)$. Both $c_\tau$ and its underlying smooth
structure are \emph{canonical}; if the first Chern class of $c_\tau$
vanishes, then $(P,\varLambda)$ and its induced omniorientation are
\emph{special}.
\end{definition}
\begin{remarks}\hfill\\
\indent (1) Although the smooth $T^m$-structure on $\zp$ is unique,
we do not know whether the same is true for the canonical smooth
structure on $M$.

(2) Reversing the orientation of $P$ acts by negating the complex
structure $c_\tau$, whereas negating the column $\lambda_i$ acts by
conjugating the complex line bundle $\rho_i$, for $1\leq i\leq m$.
None of these operations affects the validity of Condition $(*)$ for
$\varLambda$.
\end{remarks}

\begin{example}\label{toricvar}
  By applying an appropriate affine transformation to any simple
  polytope $P\subset\bR^n$, it is possible to locate its initial
  vertex at the origin, and realise the inward pointing normal vectors
  to $F_1$, \dots, $F_n$ as the standard basis.
  A small perturbation of the defining inequalities then ensures that
  the normals to the remaining $m-n$ facets are integral, without
  changing the combinatorial type of~$P$.  Any such sequence of normal
  vectors form the columns of an integral $n\times m$ matrix $N(P)$,
  whose transpose reduces to $A_P$ of \eqref{cdiz} after normalising
  the rows. For certain \emph{unsupportive} $P$ (such as the dual of a
  cyclic polytope with many vertices \cite[Nonexamples
  1.22]{da-ja91}), $N(P)$ can never satisfy Condition $(*)$. On the
  other hand, if $P$ is \emph{supportive} then $(P,N(P))$ is a
  combinatorial quasitoric pair, and the corresponding quasitoric
  manifold is a \emph{non-singular projective toric
    variety}~\cite{b-p-r07}.

  The most straightforward example is provided by the $n$-simplex
  $P=\varDelta$, whose standard embedding in $\bR^n$ gives rise to the
  $n\times(n+1)$ matrix $N(\varDelta)=(I_n:-1)$, where $-1$ denotes
  the column vector $(-1,\dots,-1)^t$. The corresponding quasitoric
  manifold $M(\varDelta,N(\varDelta))$ is the complex projective
  space $\bC P^n$, equipped with the stabilisation of the standard
  complex structure $\tau(\bC P^n)\oplus\bC\cong\oplus^{n+1}\zeta$,
  and the standard $T^n$-action
  $t\cdot[z_0,\dots,z_n]=[z_0,t_1z_1,\dots,t_nz_n]$.
\end{example}

We now reverse the above procedures, and start with an omnioriented
quasitoric manifold $(M^{2n},a)$. By definition, there is a
projection $\pi\colon M\to P$ onto a simple $n$-dimensional polytope
$P=P(M)$, which is oriented by the underlying orientation of $M$. We
continue to denote the facets of $P$ by $F_1$, \dots, $F_m$, so that
the facial submanifolds $M_i\subset M$ are given by $\pi^{-1}(F_i)$
for $1\le i\le m$, and the corresponding normal bundles $\nu_i$ are
oriented by the omniorientation of $M$. The isotropy subcircles
$T(M_i)<T^n$ are therefore specified unambiguously, and determine
$m$ primitive vectors in $\bZ^n$; these form the columns of an
$n\times m$ matrix $\varLambda'$, which automatically satisfies
Condition $(*)$. However, condition $(1)$ of Definition \eqref{data}
can only hold if the $F_i$ are finely ordered, in which case the
matrix $\varLambda(M)\letbe L\varLambda'$ is refined, for an
appropriate choice of $L$ in $GL(n,\bZ)$. So $(P(M),\varLambda(M))$ is a
quasitoric combinatorial pair, which we call the \emph{combinatorial
data} associated to $M$.

The omnioriented quasitoric manifold of the pair
$(P(M),\varLambda(M))$ is \emph{equivalent} to $(M,a)$ in the
sense of \cite{da-ja91}, being $T^n$-equivariantly homeomorphic up
to an automorphism of $T^n$. The automorphism is specified by the
choice of $L$ above. On the other hand, the combinatorial data
associated to $M(P,\varLambda)$ is $(P,\varLambda)$ itself, unless
an alternative fine ordering is chosen. In that case, the new
ordering differs from the old by a permutation $\sigma$, with the
property that $F_{\sigma^{-1}(1)}\cap\dots\cap F_{\sigma^{-1}(n)}$
is non-empty. If we interpret $\sigma$ as an $m\times m$
permutation matrix $\varSigma$, then the resulting combinatorial
data is $(P',L\varLambda\hspace{.5pt}\varSigma)$, which we deem to
be \emph{coincident with} $(P,\varLambda)$. We may summarise this
analysis as follows.
\begin{theorem}\label{corre}
  \daja\ equivalence classes of omnioriented quasitoric manifolds
  $(M,a)$ correspond bijectively to coincidence classes of combinatorial
  quasitoric pairs $(P,\varLambda)$.
\end{theorem}

We are now in a position to fulfill our promise of expressing
equivariant properties of $M$ in terms of its combinatorial data
$(P,\varLambda)$. Every fixed point $x$ takes the form
$M_{j_1}\cap\dots\cap M_{j_n}$ for some sequence $j_1<\dots<j_n$,
and it will be convenient to write $\varLambda_x$ for the $n\times
n$ submatrix $(\lambda_{j_1} \dots \lambda_{j_n})$.
\begin{proposition}\label{combdesc1}
An omniorientation of $M$ is special if and only if the column sums
of $\varLambda$ satisfy $\sum_{i=1}^n\lambda_{i,j}=1$, for every
$1\leq j\leq m$.
\end{proposition}
\begin{proof}
By \cite[Theorem~4.8]{da-ja91}, $H^2(M;\bZ)$ is the free abelian
group generated by elements $u_j\letbe c_1(\rho_j)$ for $1\leq
j\leq m$, with $n$ relations $\varLambda u=0$. These give
\[
u_i+\sum_{j=n+1}^m\!\lambda_{i,j}u_j\;=\;0
\]
for $1\leq i\leq n$, whose sum is
\begin{equation}\label{suhelp1}
\sum_{i=1}^nu_i+
\sum_{j=n+1}^m\!\bigg(\sum_{i=1}^n\lambda_{i,j}\bigg)u_j\;=\;0\,.
\end{equation}
Furthermore, \eqref{tansumrho} implies that
$c_1(c_\tau)=\sum_{j=1}^mu_j$, so the omniorientation is special if
and only if $\sum_{j=1}^mu_j=0$. Subtracting \eqref{suhelp1} gives the
required formulae.
\end{proof}

Before proceeding to the fixed point data for $(M,a)$, we reconcile
Definition~\ref{defsign} with notions of sign that appear in earlier
works on toric topology. For any fixed point $x$ the decomposition
\eqref{repre} arises by splitting $\tau_x(M)$ into subspaces normal
to the $M_{j_k}$, leading to an isomorphism
\begin{equation}\label{2orie}
  \tau_x(M)\;=\;(\nu_{j_1}\oplus\dots\oplus\nu_{j_n})|_x
  \;\cong\;(\rho_{j_1}\oplus\dots\oplus\rho_{j_n})|_x\,.
\end{equation}
\begin{lemma}\label{qts}
  The sign of any fixed point $x\in M$ is given by: $\varsigma(x)=1$
  if the orientation of $\tau_x(M)$ induced by the orientation of $M$
  coincides with that induced by the orientations of
  the $\rho_{j_k}$ in \eqref{2orie}, and by $\varsigma(x)=-1$ otherwise.
\end{lemma}
\begin{proof}
By construction, the isomorphism $\nu_j\cong\rho_j$ is of
$T^n$-equivariant bundles for any $1\leq j\leq m$, and the
orientation of both is induced by the action of the isotropy
subcircle specified by $\lambda_j$. Moreover, $\rho_j$ is trivial
over $M\setminus M_j$, so $T^n$ acts non-trivially only on the
summand $(\rho_{j_1}\oplus\dots\oplus\rho_{j_n})|_x$. It follows
that the composite map of Definition \ref{defsign} agrees with the
isomorphism
\begin{equation}\label{ormap}
  \tau_x M\stackrel{\cong}{\lra}(\rho_{j_1}\oplus\ldots\oplus\rho_{j_n})|_x
\end{equation}
of \eqref{2orie}, and that $\varsigma(x)$ takes values as claimed.
\end{proof}
Lemma \ref{qts} confirms that Definition~\ref{defsign} is
equivalent to those of \cite[\S5]{b-p-r07}, \cite{dobr01},
\cite[\S4]{masu99}, and \cite{pano01} in the case of quasitoric
manifolds.
\begin{theorem}\label{signswtscomb}
  For any quasitoric manifold $M$ with combinatorial data
  $(P,\varLambda)$ and fixed point $x$, let $N(P)_x$ be a matrix of
  column vectors normal to $F_{j_1}$, \dots, $F_{j_n}$, and $W_x$ be
  the matrix determined by $W_x^t\varLambda_x=I_n$; then
\begin{enumerate}
\item
the sign $\varsigma(x)$ is given by $\mathop{\rm
sign}\big(\det(\varLambda_x\hspace{.2pt}N(P)_x)\big)$
\item the weight vectors $w_1(x)$, \dots $w_n(x)$ are the columns of
$W_x$.
\end{enumerate}
\end{theorem}
\begin{proof}
(1) In order to write the isomorphism \eqref{ormap} in local
coordinates, recall \eqref{tansumrho} and \eqref{orientM}. So
$\tau_x(M)$ is isomorphic to $\bR^{2n}$ by choosing coordinates for
$P$ and $\tau_1(T^n)$, and $(\rho_1\oplus\dots\oplus\rho_m)|_x$ is
isomorphic to $\bR^{2m}$ by choosing coordinates for $\bR^m$ and
$\tau_1(T^m)$. Then the composition
\[
  \tau_x(M)\lra\tau_x(M)\oplus\bR^{2(m-n)}
  \stackrel{c_\tau}\longrightarrow
  (\rho_1\oplus\ldots\oplus\rho_m)|_x
\]
is represented by the $2m\times2n$--matrix
\[
\begin{pmatrix}
  A_P&0\\
  0&\varLambda^t
\end{pmatrix}\,,
\]
and~\eqref{ormap} is given by restriction to the relevant rows. The
result follows for $A_P^t$, and hence for any $N(P)$.

(2) The formula follows directly from the definitions, as in
\cite[\S1]{pano01}.
\end{proof}

\begin{remarks}
Theorem \ref{signswtscomb}(1) gives the following algorithm for
calculating $\varsigma(x)$: write the column vectors of
$\varLambda_x$ in such an order that the inward pointing normal
vectors to the corresponding facets give a positive basis of
$\bR^n$, and calculate the determinant of the resulting $n\times n$
matrix. This is equivalent to the definition of $\varsigma(x)$ in
\cite[\S5.4.1]{bu-pa02} and \cite[\S1]{pano01}. Theorem
\ref{signswtscomb}(2) states that $w_1(x)$, \dots, $w_n(x)$ and
$\lambda_{j_1}$, \dots, $\lambda_{j_n}$ are conjugate bases of
$\bR^n$.
\end{remarks}
Inserting the conclusions of Theorem \ref{signswtscomb} into
Corollaries \ref{evutg} and \ref{evgenus} completes our combinatorial
evaluation of genera on $(M,a)$. Of course the Conner-Floyd relations
stemming from the latter must be consequences of conditions (1) and
(2) of Definition~\ref{data} on the minors of $\varLambda$, and it is
interesting to speculate on the extent to which this implication is
reversible.

\begin{example}\label{cqpsplx}
Let $P$ be the $n$--simplex $\varDelta$, geometrically represented
by the standard simplex in $\bR^n$; it is oriented and finely
ordered by the standard basis. So $\varLambda$ takes the form
$(I_n:\epsilon)$, where $\epsilon$ denotes a column vector
$(\epsilon_1,\dots,\epsilon_n)^t$, and Condition $(*)$ confirms
that $(\varDelta,(I_n:\epsilon))$ is a combinatorial quasitoric
pair if and only if $\epsilon_i=\pm 1$ for every $1\leq i\leq n$.
When $\epsilon_i=-1$ for all $1\leq i\leq n$, the corresponding
omnioriented quasitoric manifold $M=M(\varDelta,(I_n:\epsilon))$
is the toric variety $\bC P^n$ of Example \ref{toricvar}. In
general, $M$ is obtained from $\bC P^n$ by a change of
omniorientation, in which the $i$th summand of the stable normal
bundle is reoriented when $\epsilon_i=1$, and unaltered otherwise;
it is denoted by $\bC P^n_\epsilon$ when convenient.

If the facets of $\varDelta$ are reordered by a transposition
$(j,n+1)$, the resulting pair $(\varDelta',(I_n:\epsilon')$ is
coincident with $(\varDelta,(I_n:\epsilon))$. A simple calculation
reveals that $\epsilon'_i=-\epsilon_i\epsilon_j$ for $i\neq j$,
whereas $\epsilon_j'=\epsilon_j$.

The $n+1$ fixed points correspond to the vertices of $\varDelta$,
and they are most conveniently labelled $x_0$, \dots, $x_n$, with
$x_0$ initial. The submatrix $(I_n:\epsilon)_{x_0}$ is $I_n$, and
the submatrices $(I_n:\epsilon)_{x_k}$ are obtained by deleting
the $k$th column, for $1\leq k\leq n$. So Theorem
\ref{signswtscomb}(2) expresses the weight vectors by
$w_j(x_0)=e_j$, and
\[
w_j(x_k)\;=\;
\begin{cases}
e_j-\epsilon_j\epsilon_ke_k&\text{for $1\leq j<k$}\\
e_{j+1}-\epsilon_{j+1}\epsilon_ke_k&\text{for $k\leq j<n$}\\
\epsilon_ke_k&\text{for $j=n$}
\end{cases}
\]
for $1\leq j,\,k\leq n$. Applying Corollary \ref{evgenus} for the
augmentation genus $ag$ then gives
\begin{equation}\label{agcpeps}
\frac{\varsigma(x_0)}{u_1\dots u_n}\;+\;
  \sum_{k=1}^n\Bigg(\frac{\varsigma(x_k)}{\epsilon_ku_k}
  \prod_{j\neq k}\frac1{(u_j-\epsilon_j\epsilon_ku_k)}\Bigg)
  \;=\;0,
\end{equation}
from which it follows that
$\epsilon_i=-\varsigma(x_i)/\varsigma(x_0)$ for $1\le i\le n$.
\end{example}

Formula \eqref{agcpeps} is closely related to \eqref{agcpnstand} of
Example \ref{agcfrels}. This case is of interest even for $n=1$, and
Corollary \ref{evutg} evaluates the universal toric genus by
\[
  \varPhi(\bC P^1)\;=\;\frac 1u+\frac 1{\bar u}
\]
in $\varOmega^{-2}_U(BT_+)$, where $\bar u$ denotes the infinite
series $[-1](u)$.

%
%
%
%
%
%
%
%
%

\section{Applications and further examples}\label{apfuex}

In our final section we consider applications of toric genera to
questions of rigidity, and discuss additional examples that support
the theory.


\subsection{Parametrised genera}
The parametrised genus $\varPhi(\pi)$ is an invariant of
tangentially stably complex $T^k$-bundles. We now return to Examples
\ref{isfipoexas}, and apply the formula \eqref{defgpsi} to complete
the calculations. The results suggest that there exists a systematic
theory of genera of bundles, with interesting applications to
rigidity phenomena.

\begin{example}\label{ptg1}
  The $T^3$-bundle $\bC P^1\ra U(3)/T^3\stackrel{\pi}{\ra}\bC P^2$ of
  Example \ref{isfipoexas}(1) is tangentially complex, and
  $\varPhi(\pi)$ lies in $\varOmega^{-2}_U(\EP_+)$, where $\EP$
  denotes $ET^3\times_{T^3}\bC P^2$.
  Since $\EP$ is actually $\bC
  P(\eta_1\times\eta_2\times\eta_3)$ over $BT^3$, there is an
  isomorphism
\[
  \varOmega^*_U(\EP_+)\;\cong\;
  \varOmega_*^U[[u_1,u_2,u_3,v]\big/\left((v+u_1)(v+u_2)(v+u_3)\right)
\]
of $\varOmega^*_U$-algebras, where $v\letbe c_1^{M\?U}(\rho)$ for
the canonical projective bundle $\rho$. So
\begin{equation}\label{ptg1form}
\varPhi(\pi)\;=\;
\sum_{\omega}g_{0,\omega}u^\omega
+\sum_{\omega}g_{1,\omega}\,vu^\omega
+\sum_{\omega}g_{2,\omega}\,v^2u^\omega
\end{equation}
holds in $\varOmega^{-2}_U(\EP_+)$ by \eqref{defgpsi}, where
$g_{j,\omega}$ lies in $\varOmega^U_{2(j+|\omega|+1)}$, and
$\omega=(\omega_1,\omega_2,\omega_3)$; in particular,
$g_{0,0}+g_{1,0}v+g_{2,0}v^2$ is the cobordism class of $\pi$ in
{$\varOmega^{-2}_U(\bC P^2_+)$}.

For $0\leq j\leq 2$, the element $a_{j,\omega}$ of Theorem
\ref{Gpirepsgpi} is represented by a map
\[
  f_{j,\omega}\colon(S^3)^\omega\times_{T^\omega}B_j \llra
  ET^3\times_{T^3}\bC P^2,
\]
which combines a representative for $b_\omega$ on the first factor
with a $T^3$-equivariant representative for
$b_j\in\varOmega^U_{2j}(\bC P^2)$ on the second. It follows that
$g_{j,\omega}$ is represented by the induced stably complex
structure on the total space of the bundle
\[
  \bC P^1\lra(S^3)^\omega\times_{T^\omega}\!\bC
  P(\psi_j^\perp\?\oplus\bC^{2-j})\lra
  (S^3)^\omega\times_{T^\omega}\!B_j,
\]
Thus $g_{0,0}=[\bC P^1]$ in $\varOmega_2^U$, as expected;
furthermore, $g_{1,0}=[\bC P(\zeta\oplus\bC)]$ in $\varOmega_4^U$
and $g_{2,0}=[\bC P(\psi_2^\perp)]$ in $\varOmega_6^U$, where
$\zeta=\psi_1^\perp$ lies over $B_1=S^2$. These classes may be
computed in terms of standard polynomial generators for
$\varOmega^U_*$ as required.
\end{example}

The properties of $\varPhi(\pi)$ may be interpreted as obstructions
to the existence of equivariant bundles $\pi$. This principle is
illustrated by the family of $4$--dimensional omnioriented
quasitoric manifolds over the combinatorial square $P=I^2$.

\begin{example}\label{tarassq}
Let the vertices of $P$ be ordered cyclically, and write
$\varLambda$ as
\[
  \begin{pmatrix}
  1&0&\epsilon_1&\delta_2\\0&1&\delta_1&\epsilon_2
  \end{pmatrix};
\]
so $(P,\varLambda)$ is a combinatorial quasitoric pair if and only
if the equations
\begin{equation}\label{qtpsqeq}
  \epsilon_1=\pm1,\quad\epsilon_2=\pm1,\quad
  \epsilon_1\epsilon_2-\delta_1\delta_2=\pm1
\end{equation}
are satisfied. The fixed points $x_i$ of $M(P^2,\varLambda)$
correspond to the vertices of $P^2$, and are determined by the pairs
$(i,i+1)$ of facets, where $1\le i\le 4$ (and $x_5=x_1$). We assume
that $\varsigma(x_1)=1$, by reversing the global orientation of $P^2$
if necessary. Then the calculations of Lemma~\ref{qts} determine the
remaining signs by
\[
  \varsigma(x_2)=-\epsilon_1,\quad
  \varsigma(x_3)=\epsilon_1\epsilon_2-\delta_1\delta_2,
  \quad\varsigma(x_4)=-\epsilon_2,
\]
and a criterion of Dobrinskaya ~\cite[Corollary~7]{dobr01}
confirms that $M(P^2,\varLambda)$ is the total space of a
tangentially complex $T^2$-equivariant bundle precisely when
$\delta_1\delta_2=0$.

The necessity of this condition follows from Theorem~\ref{evutgx}, by
considering the equations associated to the augmentation genus $ag$.
Given \eqref{qtpsqeq}, the contributions $k(x_i)$ of the four fixed
points must always sum to zero; and if $M(P,\varLambda)$ is the total
space of some $T^2$-bundle, then the $k(x_i)$ must split into pairs,
which also sum to zero. But $k(x_1)+k(x_3)\neq 0$ by inspection, so
two possibilities arise.

Firstly, $k(x_1)+k(x_4)$ and $k(x_2)+k(x_3)$ are given by
\[
  \frac1{t_1t_2}+
  \frac{-\epsilon_2}{\epsilon_2t_2(t_1-\epsilon_2\delta_2t_2)}
\spandsp
  \frac{-\epsilon_1}{\epsilon_1t_1(-\epsilon_1\delta_1t_1+t_2)}+
  \frac{\epsilon_1\epsilon_2}
  {\epsilon_1t_1(-\epsilon_1\epsilon_2\delta_1t_1+\epsilon_2t_2)}
\]
respectively, and vanish precisely when $\delta_2=0$. Secondly,
and by similar reasoning, $k(x_1)+k(x_2)$ and $k(x_3)+k(x_4)$
vanish precisely when $\delta_1=0$.
\end{example}

Example \ref{tarassq} extends inductively to quasitoric manifolds
over higher dimensional cubes, and shows that any such manifold
splits into a tower of fibre bundles if and only if the
corresponding submatrix $\varLambda_\star$ may be converted to
lower triangular form by simultaneous transpositions of rows and
columns. Compare~\cite[Proposition~3.2]{ma-pa08}.

\subsection{Elliptic genera}
Since the early 1980s, cobordism theory has witnessed the
emergence of \emph {elliptic genera}, taking values in rings of
modular forms. The original inspiration was index-theoretical, and
driven by Witten's success \cite{witt87} in adopting the
Atiyah-Bott fixed-point formula to compute the equivariant index
of a putative Dirac operator on the free loop space of a Spin
manifold. Analytic and formal group-theoretic versions of complex
elliptic genera were then proposed by Ochanine \cite{ocha87}, in
the oriented category. These led directly to the homotopical
constructions of Landweber, Ravenel, and Stong \cite{larast95},
Hopkins \cite{hopk95}, and Ando, Hopkins, and Strickland
\cite{a-h-s01}, amongst others; their constructions include
several variants of elliptic cohomology, such as the complex
oriented theory represented by the spectrum $Ell$.

Mindful of Definition \ref{defhtpclgenus} and Examples
\ref{cgenexs}, we shall focus on the homotopical elliptic genus
$(t^{Ell},h^{Ell})$, where $Ell_*$ is the graded ring
$\bZ[\frac{1}{2}][\delta,\epsilon,(\delta^2-\epsilon)^{-1}]$, and
$\delta$ and $\epsilon$ have dimensions $4$ and $8$ respectively.
Analytically, it is determined by the elliptic integral
\begin{equation}\label{elllog}
  m^{H\wedge E\?ll}(x)\;=\;\int_0^{x} \frac{dt}{\sqrt{1-2\delta
  t^2+\epsilon t^4}}
\end{equation}
over $(H\wedge Ell)_*$, so $b^{H\wedge Ell}(x)$ is a Jacobi sn
function modulo constants.

We shall discuss certain generalisations of $t^{Ell}_*$ in the toric
context. Their homotopical status remains unresolved, and it is most
convenient to work with the analytical and formal group theoretic
methods used in their construction. Nevertheless, we proceed with
the extension \eqref{niceext} closely in mind, and plan to describe
integrality and homotopical properties of these genera elsewhere.

Following~\cite{buch90}, we consider the formal group law
\begin{equation}\label{bsfgl}
F_b(u_1,u_2)\;=\; u_1c(u_2)+u_2c(u_1)-au_1u_2-
\frac{d(u_1)-d(u_2)}{u_1c(u_2)-u_2c(u_1)}u^2_1u^2_2
\end{equation}
over the graded ring $R_*=\bZ[a,c_j,d_k\colon j\geq 2, k\geq 1]/J$,
where $\deg a=2$, $\deg c_j=2j$ and $\deg d_k=2(k+2)$; also, $J$ is
the ideal of associativity relations for \eqref{bsfgl}, and
\[
c(u)\;\letbe\;1+\sum_{j\geq 2}c_ju^j\sands d(u)\;\letbe\;\sum_{k\geq
1}d_ku^k.
\]
\begin{theorem}\label{bsexp}
The exponential series $f_b(x)$ of \eqref{bsfgl} may be written
analytically as $\exp(ax)/\phi(x,z)$, where
\begin{equation}\label{bakakh}
\phi(x,z)\;=\;
\frac{\sigma(z-x)}{\sigma(z)\sigma(x)}\exp(\zeta(z)\,x),
\end{equation}
$\sigma(z)$ is the Weierstrass sigma function, and
$\zeta(z)=(\ln\sigma(z))'$.

Moreover, $R_*\otimes\bQ$ is isomorphic to $\bQ[a,c_2,c_3,c_4]$ as
graded algebras.
\end{theorem}
\begin{proof}
The formal group law may be defined by
\begin{equation}\label{fglexplog}
f_b\bigl(f^{-1}_{b}(u_1)+f^{-1}_{b}(u_2)\bigr)\;=\;F_{b}(u_1,u_2)
\end{equation}
over $R_*\otimes\bQ$, so the exponential series satisfies the
functional equation
\begin{multline*}
  f(u_1+u_2)\;=\;f(u_1)c(f(u_2))+f(u_2)c(f(u_1))-af(u_1)f(u_2)\\-
  \frac{d(f(u_1))-d(f(u_2))}{f(u_1)c(f(u_2))-f(u_2)c(f(u_1)}f^2(u_1)f^2(u_2)
\end{multline*}
over $R\otimes\bC$. Substituting
$\xi_1(u)=c^2(f(u))-af(u)c(f(u))+d(f(u))f^2(u)$ and $\xi_2(u)=c(f(u))$
gives
\[
  f(u_1+u_2)=\frac{f^2(u_1)\xi_1(u_2)-f^2(u_2)\xi_1(u_1)}
  {f(u_1)\xi_2(u_2)-f(u_2)\xi_2(u_1)}.
\]
This equation is considered in~\cite{buch90}, and the required form
of the solution is established by~\cite[Theorem~1 and
Corollary~8]{buch90}.

Because $F_b$ is commutative, its classifying map $\varOmega^U_*\ra
R_*$ is epimorphic, and $R_*\otimes\bQ$ is generated by the
coefficients of $f_{b}(x)=x+\sum_{k\geq 1}f_kx^{k+1}$. In order to
identify polynomial generators in terms of the elements $a$, $c_j$
and $d_k $, it therefore suffices to reduce equation
\eqref{fglexplog} modulo decomposable coefficients, and consider
\begin{multline*}
\sum_{k\geq
1}f_k\big((u_1+u_2)^{k+1}-u_1^{k+1}-u_2^{k+1}\big)\\\equiv\;
-au_1u_2+\sum_{k\geq 2}c_k\big(u_1u_2^k+u_1^ku_2\big)-
u_1^2u_2^2\sum_{k\geq 1}d_k(u_1^k-u_2^k)/(u_1-u_2).
\end{multline*}
Equating coefficients in degrees $4$, $6$, $8$, and $10$ then gives
\begin{equation}\label{fks}
  2f_1=-a,\;\; 3f_2\equiv c_2,\;\; 4f_3\equiv c_3,\;\;
  6f_3\equiv -d_1,\;\;5f_4\equiv c_4,\;\; 10f_4\equiv -d_2;
\end{equation}
in higher degrees the only solutions are $f_k\equiv 0$, so $f_k$ is
decomposable for $k\geq 5$.
\end{proof}

The function $\varphi(x,z)$ of \eqref{bakakh} is known as the
\emph{Baker--Akhiezer function} associated to the elliptic curve
$y^2=4x^3-g_2x-g_3$. It satisfies the \emph{Lam\'e equation}, and is
important in the theory of nonlinear integrable equations.
Krichever~\cite{kric90} studies the genus $kv$ corresponding to the
exponential series $f_b$, which therefore classifies the formal
group law~\eqref{bsfgl}. Analytically, it depends on the four
complex variables $z$, $a$, $g_2$ and $g_3$.

\begin{corollary}\label{kviso}
  The genus $kv\colon\varOmega_*^U\ra R_*$ induces an isomorphism of
  graded abelian groups in dimensions $<10$.
\end{corollary}
\begin{proof}
  By \eqref{fks}, $kv\otimes 1_\bQ$ is monomorphic in dimensions
  $<10$; so $kv$ is also mono\-morphic, because $\varOmega^U_*$ is
  torsion-free. Since $kv$ is epimorphic in all dimensions, the result
  follows.
\end{proof}

Many of the genera considered in previous sections are special cases
of $kv$, and the corresponding methods of calculation may be adapted
as required.
\begin{examples}\label{specialcasesbs}\hfill\\
\indent(1) The Abel genus $ab$ and $2$-parameter Todd genus
$\mathit{t2}$ of Examples \ref{cgenexs}(5) may be identified with
the cases
\[
d(u)=0,\;\;a=y+z\spandsp c(u)=1-yzu^2,\;\;d(u)=-yz(y+z)u-y^2z^2u^2
\]
respectively. The corresponding formal group laws are
\[
F_{ab}=u_1c(u_2)+u_2c(u_1)-(y+z)u_1u_2 \spandsp
F_{\mathit{t2}}=\frac{u_1+u_2-(y+z)u_1u_2}{1-yzu_1u_2}\,.
\]

(2) The elliptic genus $t^{Ell}_*$ of \eqref{elllog} corresponds to
Euler's formal group law
\[
F_{Ell}(u_1,u_2)=\frac{u_1c(u_2)+u_2c(u_1)}{1-\varepsilon
u_1^2u_2^2} =u_1c(u_2)+u_2c(u_1)+ \varepsilon
\frac{u_1^2-u_2^2}{u_1c(u_2) -u_2c(u_1)}u_1^2u_2^2\,,
\]
and may therefore be identified with the case
\[
a=0,\;\; d(u)=-\varepsilon u^2,\sands c^2(u)=R(u)\letbe 1-2\delta
u^2+\varepsilon u^4.
\]
\end{examples}

\begin{theorem}\label{suprop}
Let $M^{2n}$ be a specially omnioriented quasitoric manifold; then\\
\indent
\emph{(1)} the Krichever genus $kv$ vanishes on $M^{2n}$\\
\indent
\emph{(2)} $M^{2n}$ represents $0$ in $\varOmega^U_{2n}$ whenever
$n<5$.
\end{theorem}
\begin{proof}
(1) Let $T<T^n$ be a generic subcircle, determined by a primitive
integral vector $\nu=(\nu_1,\dots,\nu_n)$. By
\cite[Theorem~2.1]{pano01}, the fixed points $x$ of $T$ are
precisely those of $T^n$, and have $T$-weights $w_j(x)\cdot\nu$ for
$1\leq j\leq n$. By Proposition~\ref{combdesc1} and
Theorem~\ref{signswtscomb}, the sum $s(x)\letbe\sum_jw_j(x)\cdot\nu$
reduces to $\sum_j\nu_j$, which is independent of $x$ and may be
chosen to be nonzero. Since $M$ is specially omnioriented, it
follows from \cite[Theorem~2.1]{kric90} that $kv(M)=0$.

(2) Combining (1) with Corollary \ref{kviso} establishes (2).
\end{proof}

Krichever \cite[Lemma 2.1]{kric90} proves that $s(x)$ is
independent of $x$ in more general circumstances.

\begin{conjecture}\label{qtsuconjec}
Theorem \ref{suprop}(2) holds for all $n$.
\end{conjecture}

\subsection{Rigidity}
We now apply $kv$, and $\mathit{t2}$ to characterise genera that
are $T^k$-rigid on certain individual stably complex manifolds
$M$, as specified by Definition~\ref{genusdefs}. Up to complex
cobordism, we may insist that $M$ admits a tangentially stably
complex $T^k$-action with isolated fixed points $x$, and write the
fixed point data as $\{\varsigma(x),w_j(x)\}$. We continue to work
analytically.

\begin{proposition}\label{mulgenfuneq}
For any analytic exponential series $f$ over a $\bQ$-algebra $A$,
the genus $\ell_f$ is $T^k$-rigid on $M$ only if the functional
equation
\[
\sum_{\fix(M)}\!\varsigma(x)\prod_{j=1}^n
  \frac1{f(w_j(x)\cdot u)}\;=\;c
\]
is satisfied in $A[[u_1,\dots,u_k]]$, for some constant $c\in A$.
\end{proposition}
\begin{proof}
Because $A$ is a $\bQ$-algebra, $f$ is an isomorphism between the
additive formal group law and $F_f(u_1,u_2)$. So the formula follows
from Corollary \ref{evgenus}.
\end{proof}

The quasitoric examples $\bC P^1$, $\bC P^2$, and the
$T^2$-manifold $S^6$ are all instructive.
\begin{example}
A genus $\ell_f$ is $T$-rigid on ~$\mathbb C P^1$ with the standard
complex structure only if the equation
\begin{equation}\label{cp1rig}
  \frac 1{f(u)}+\frac 1{f(-u)}\;=\;c,
\end{equation}
holds in $A[[u]]$. The general analytic solution is of the form
\[
f(u)\;=\;\frac{u}{q(u^2)+cu/2},\sts{where}q(0)\;=\;1.
\]
The Todd genus of Example \ref{cgenexs}(3) is defined over $\bQ[z]$
by $f_{td}(u)=(e^{zu}-1)/z$, and \eqref{cp1rig} is satisfied with
$c=-z$. So $td$ is $T$-rigid on $\bC P^1$, and $q(u^2)$ is given by
\[
(zu/2)\cdot\frac{e^{zu/2}+e^{-zu/2}}
  {e^{zu/2}-e^{-zu/2}}
\]
in $\bQ[z][[u]]$. In fact $td$ is multiplicative with respect to
$\bC P^1$ by \cite{hirz66}, so rigidity also follows from Section
\ref{geri}.
\end{example}

\begin{example}\label{excp2e}
A genus $\ell_f$ is $T^2$-rigid on the stably complex manifold
~$\mathbb C P^2_{(1,-1)}$ of Example \ref{cqpsplx} only if the
equation
\[
\frac{1}{f(u_1)f(u_2)}-\frac{1}{f(u_1)f(u_1+u_2)}+
\frac{1}{f(-u_2)f(u_1+u_2)}\;=\;c
\]
holds in $A[[u_1,u_2]]$. Repeating for $\bC P^2_{(-1,1)}$ (or
reparametrising $T^2$) gives
\begin{equation}\label{help1}
\frac{1}{f(u_2)f(u_1)}-\frac{1}{f(u_2)f(u_1+u_2)}+
\frac{1}{f(-u_1)f(u_1+u_2)}\;=\;c\,,
\end{equation}
and subtraction yields
\[
\left(\frac{1}{f(u_1)}+\frac{1}{f(-u_1)}\right)\frac{1}{f(u_1+u_2)}
\;=\;
\left(\frac{1}{f(u_2)}+\frac{1}{f(-u_2)}\right)\frac{1}{f(u_1+u_2)}\,.
\]
It follows that
\[
\frac{1}{f(u)}+\frac{1}{f(-u)}\;=\;c'\spandsp \frac{1}{f(-u)}\;=\;
c'-\frac{1}{f(u)}
\]
in $A[[u]]$, for some constant $c'$. Substituting in \eqref{help1}
gives
\[
\left(\frac{1}{f(u_1)}+\frac{1}{f(u_2)}-c'\right)\frac{1}{f(u_1+u_2)}\;=\;
\frac{1}{f(u_1)f(u_2)}-c\,,
\]
which rearranges to
\[ f(u_1+u_2)\;=\;
\frac{f(u_1)+f(u_2)-c'f(u_1)f(u_2)}{1-cf(u_1)f(u_2)}\,.
\]
So $f$ is the exponential series of the formal group law
$F_{\mathit{t2}}(u_1,u_2)$, with $c'=y+z$ and $c=yz$.
\end{example}
Example \ref{excp2e} provides an alternative proof of a recent
result of Musin \cite{musi09}.
\begin{proposition}\label{musin}
The $2$-parameter Todd genus $\mathit{t2}$ is universal for rigid genera.
\end{proposition}
\begin{proof}
Krichever shows \cite{kric76} that $\mathit{t2}$ is rigid on
tangentially almost complex $T$-manifolds, by giving a formula for
the genus in terms of fixed point sets of arbitrary dimension. His
proof automatically extends to the stably complex situation by
incorporating the signs of isolated fixed points.

The canonical stably complex structure on $\bC P^2_{(1,-1)}$ is
$T^2$-invariant, so universality follows from Example
\ref{excp2e}(2) by restricting to arbitrary subcircles
\smash{$T<T^2$}.
\end{proof}

\begin{example}\label{fueqssix}
A genus $\ell_f$ is $T^2$-rigid on the almost complex manifold $S^6$
of Examples \ref{isfipoexas}(2) only if the equation
\[
  \frac{1}{f(u_1)f(u_2)f(-u_1-u_2)}+
  \frac{1}{f(-u_1)f(-u_2)f(u_1+u_2)}\;=\;c
\]
holds in $A[[u_1,u_2]]$, for some constant $c$. Rearranging as
\[
  \frac{f(u_1+u_2)}{f(u_1)f(u_2)}-
  \frac{-f(-u_1-u_2)}{(-f(-u_1))(-
f(-u_2))}
  \;=\;
cf(u_1+u_2)f(-u_1-u_2)
\]
and applying the differential operator $D=\frac\partial{\partial
u_1}-\frac{\partial}{\partial u_2}$ yields
\begin{equation}\label{pde}
f(u_1+u_2)D\Bigl(\frac1{f(u_1)f(u_2)}\Bigr)\;=\;
-f(-u_1-u_2)D\Bigl(\frac1{(-f(-u_1))(-f(-u_2))}\Bigr).
\end{equation}
Substituting
\[
  \phi_1(u_1+u_2)=\frac{f(u_1+u_2)}{-f(-u_1-u_2)},\quad
  \phi_2(u)=\frac1{-f(-u)},\quad
  \phi_4(u)=\frac1{f(u)}
\]
into \eqref{pde} then gives
\[
  \phi_1(u_1+u_2)=
  \frac{\phi'_2(u_1)\phi_2(u_2)-\phi_2(u_1)\phi'_2(u_2)}
  {\phi'_4(u_1)\phi_4(u_2)-\phi_4(u_1)\phi'_4(u_2)}\,,
\]
which is exactly the equation~\cite[(1)]{br-bu97}, with
$\phi_3(u)=\phi_2'(u)$ and $\phi_5(u)=\phi_4'(u)$. So by
~\cite[Theorem~1]{br-bu97}, the general analytic solution is of
the form $\exp(ax)/\phi(x,z)$, and $f$ coincides with the
exponential series $f_b$ of Theorem \ref{bsexp}.
\end{example}

\begin{theorem}\label{kvrigid}
The Krichever genus $kv$ is universal for genera that are rigid on
$SU$-manifolds.
\end{theorem}
\begin{proof}
The rigidity of $kv$ on almost complex manifolds with zero first
Chern class is established in~\cite{kric90}; the proof extends
immediately to general $SU$-manifolds, by incorporating signs as in
Proposition \ref{musin}.

On $S^6$, the almost complex $T^2$-invariant structure of Example
\ref{isfipoexas}(2) is stably $SU$, so universality follows from
Example \ref{fueqssix} by restricting to arbitrary subcircles
$T<T^2$.
\end{proof}
An alternative characterisation of Krichever's genus is given by
Totaro \cite[Theorem 6.1]{tota00}, in the course of proving that
$kv$ extends to singular complex algebraic varieties. Building on
work of H\"ohn \cite{hohn91}, Totaro interprets $kv$ as universal
amongst genera that are invariant under \emph{$SU$-flops}; this
property is equivalent to $SU$-fibre multiplicativity with respect
to a restricted class of stably tangentially $SU$ bundles. A
second version of the latter approach is dealt with by Braden and
Feldman \cite{br-fe05}.


\subsection{Realising fixed point data}

In 1938, PA Smith considered actions of finite groups $G$ on the
sphere. In the case of two fixed points, he asked if the tangential
representations necessarily coincide; this became a celebrated
question that now bears his name, although the answer is known to be
negative for certain $G$. It suggests the following realisation
problem for torus actions.
\begin{problem}\label{realise}
  For any set of signs $\varsigma(x)$ and matching weight vectors
  $w_j(x)$, find necessary and sufficient conditions for the existence
  of a tangentially stably complex $T^k$-manifold with the given
  fixed point data.

\end{problem}

This problem was originally addressed by Novikov \cite{novi67} for
$\bZ/p$ actions, and his emphasis on the necessity of the
Conner-Floyd relations initiated much of the subsequent output of
the Moscow school. It was also a motivating factor for tom Dieck
\cite{tomd70}, who provided a partial answer in terms of the
$K$-theoretic Hurewicz genus and the Hattori-Stong theorem. We have
introduced the relations in Proposition \ref{cfrels}, and the analytic
viewpoint provides an elegant way of combining their implications
under the action of an appropriate genus.
\begin{theorem}\label{genuscf} For any exponential series $f$ over
a $\bQ$-algebra $A_*$, the formal Laurent series
\[
  \sum_{\fix(M)}\!\varsigma(x)\prod_{j=1}^n
  \frac1{f(w_j(x)\cdot u)}
\]
lies in $A[[u_1,\dots,u_k]]$, and has zero principal part.
\end{theorem}
As $f$ varies over formal exponential series, Theorem
\eqref{genuscf} provides strong restrictions on the fixed point
data.

\begin{example}
Applying Example \ref{agcfrels} in case $E=M$ and $X=*$ shows that
the augmentation genus $ag$ constrains the fixed point data of $M$
to satisfy the condition
\[
  \sum_{\fix(M)}\!\varsigma(x)\prod_{j=1}^n
  \frac1{w_j(x)\cdot u}\;=\;0
\]
in $\bZ[[u_1,\dots,u_k]]$. Homotopical versions of this description
are implicit in Remarks \ref{3rems} and Example \ref{agcfrels}.
\end{example}

%
%
%
%
%
%
%
%
%

\end{document}